\documentclass[reqno,
dvipdfmx,  
11pt]{amsart}

\usepackage{
amssymb,
amsmath,
amsthm,
eucal,
empheq,
cases,
dsfont,
multicol,
mathrsfs,
tikz,
graphicx,
hyperref,
}
\usepackage{lipsum}
\makeatletter
\renewcommand*{\eqref}[1]{%
\hyperref[{#1}]{\textup{\tagform@{\!\!\ref*{#1}}}}%
}\makeatother 

\makeatletter
 
  \@addtoreset{equation}{section}
 \makeatother
\theoremstyle{plain}
\newtheorem{theorem}{Theorem}[section]
\newtheorem{lemma}[theorem]{Lemma}
\newtheorem{proposition}[theorem]{Proposition}
\newtheorem{corollary}[theorem]{Corollary}
\theoremstyle{definition}

\newtheorem{remark}[theorem]{Remark}
\newtheorem{example}[theorem]{Example}
\newtheorem{assumption}{Assumption}

\newcommand{\CAL}[1]{{\mathcal #1}}
\newcommand{\SCR}[1]{{\mathscr #1}}
\newcommand{\J}[1]{\left\langle #1 \right\rangle}

\newcommand{\lr}[1]{{\lceil #1 \rceil}}

\newcommand{\bignorm}[1]{{\left\|#1\right\|}}
\newcommand{\norm}[1]{{\|#1\|}}

\def\supp{\mathop{\mathrm{supp}}\nolimits}

\def\Re{\mathop{\mathrm{Re}}\nolimits}
\def\Im{\mathop{\mathrm{Im}}\nolimits}
\def\loc{\mathop{\mathrm{loc}}\nolimits}

\def\R{{\mathbb{R}}}
\def\Z{{\mathbb{Z}}}
\def\N{{\mathbb{N}}}
\def\C{{\mathbb{C}}}

\def\F{{\mathscr{F}}}

\def\<{{\langle}}
\def\>{{\rangle}}

\def\ep{{\varepsilon}}
\def\ds{\displaystyle}

\title[NLS with long-range potentials]{Modified scattering for nonlinear Schr\"odinger equations with long-range potentials}

\author{Masaki Kawamoto}
\address[M. Kawamoto]{Research Institute for Interdisciplinary Science, Okayama University, 3-1-1, Tsushimanaka, Kita-ku, Okayama City, Okayama, 700-8530, Japan}
\email{kawamoto.masaki@okayama-u.ac.jp}
\author{Haruya Mizutani}
\address[H. Mizutani]{Department of Mathematics, Graduate School of Science, Osaka University, Toyonaka, Osaka 560-0043, Japan}
\email{haruya@math.sci.osaka-u.ac.jp}
\begin{document}

\begin{abstract}
We study the final state problem for the nonlinear Schr\"{o}dinger equation with a critical long-range nonlinearity and a long-range linear potential. Given a prescribed asymptotic profile which is different from the free evolution, we construct a unique global solution scattering to the profile. In particular, the existence of the modified wave operators is obtained for sufficiently localized small scattering data. The class of potential includes a repulsive long-range potential with a short-range perturbation, especially the positive Coulomb potential in two and three space dimensions. The asymptotic profile is constructed by combining Yafaev's type linear modifier \cite{Yafaev} associated  with the long-range part of the potential and the nonlinear modifier introduced by Ozawa \cite{Ozawa1991}. Finally, we also show that one can replace Yafaev's type modifier by Dollard's type modifier under a slightly stronger decay assumption on the long-range potential. This is the first positive result on the modified scattering for the nonlinear Schr\"{o}dinger equation in the case when  both of the nonlinear term and the linear potential are of long-range type. 
\end{abstract}

\maketitle


\section{Introduction}
We consider the scattering problem for the nonlinear Schr\"odinger equation (NLS) on $\R^n$ with a linear potential of the form
\begin{align}
\label{NLS1}
i \partial _t u=Hu+F(u),\quad t\in \R,\quad x\in \R^n,
\end{align}
where $n=1,2,3$, $u=u(t,x)$ is complex-valued, 
$$
F(u)=\nu |u|^{\frac 2n} u,\quad \nu\in \R,
$$
and $H$ is the Schr\"odinger operator with a real-valued potential $V(x)$: 
$$
H=H_0+V,\quad H_0=-\frac12\Delta=-\frac12\sum_{j=1}^n\frac{\partial^2}{\partial x_j ^2}. 
$$
The purpose of the paper is to study the asymptotic behavior of solutions to the final state problem associated with \eqref{NLS1}, namely we construct a global solution of \eqref{NLS1} scattering to a prescribed asymptotic profile $u_{\mathrm{p}}(t)$ with a given scattering datum $u_+$ as $t\to \infty$: 
$$
\|u(t)-u_{\mathrm{p}}(t)\|_{L^2(\R^n)}\to 0,\quad t\to \infty. 
$$
The nonlinearity $|u|^{\frac 2n} u$ is known to be critical in the sense that, at least when $V\equiv0$, $u(t)$ does not scatter to a free solution $e^{-itH_0}u_+$, but the corresponding asymptotic profile is given by 
\begin{align}
\label{modified_profile_1}
(it)^{-n/2} e^{\frac{i|x|^2}{2t}}\widehat{u_+}\left(x/t\right)e^{-i \nu |\widehat{u_+} (x/t)|^{2/n} \log t},
\end{align}
which has an additional phase correction term $e^{-i \nu |\widehat{u_+} (x/t)|^{2/n} \log t}$ compared with the one for $e^{-itH_0}u_+$. Such a phenomenon is called the {\it long-range scattering}, or more recently  {\it modified scattering}. It is well known that the modified scattering also occurs for the linear case ($\nu\equiv0$) when $V$ decays at infinity sufficiently slowly, say $V(x)=O(|x|^{-\rho})$ as $|x|\to \infty$ with some $0<\rho\le1$. The modified scattering has been extensively studied for both of the linear Schr\"odinger equation with long-range potentials ($\nu=0$) and the standard NLS \eqref{NLS1} without linear potentials (see Section \ref{background} below). 

In the present paper, we focus on the case when both of the linear potential $V$ and the nonlinear term are of long-range type. As a typical example, our class of potentials particularly includes the repulsive inverse power potential $$V(x)=Z|x|^{-\rho}$$ with $n=2,3$, $n/4<\rho\le1$ and $Z>0$, especially the repulsive Coulomb potential $V(x)=Z|x|^{-1}$ (see Assumption \ref{assumption_A} and Example \ref{example_V} below for the precise assumption on $V$). 

To the best of our knowledge, this is the first result on the modified scattering for such a combined case. In particular, the main novelty of the paper is determining the relevant asymptotic profile $u_{\mathrm{p}}(t)$ for the scattering theory associated with \eqref{NLS1}, which turns out to be depend on both the long-range part of $V$ and the nonlinear term $F(u)$. We hope that the argument in the paper, especially the construction of the profile $u_{\mathrm{p}}(t)$, could be also applied to other problems on the long-range scattering theory such as the asymptotic completeness of the Cauchy problem for \eqref{NLS1} or the scattering theory for the nonlinear Hartree equations with long-range potentials, and so on. 

\subsection{Main result}
In order to state the main result, we need to introduce several notation and  assumptions. Recall first that a pair $(q,r) $ is said to be {\it admissible} if 
\begin{align}
\label{admissible}
2\le q,r\le\infty,\quad\frac2q+\frac nr=\frac n2,\quad (n,q,r)\neq(2,2,\infty).
\end{align}

\begin{assumption}
\label{assumption_A}
$V$ is a real-valued function satisfying the following two conditions (A1) and (A2): 
\begin{itemize}
\item[(A1)] $V$ is decomposed into three parts as $V=V^{\mathrm{S}}+V^{\mathrm{L}}+V^{\mathrm{sing}}$ with real-valued functions $V^{\mathrm{S}}, V^{\mathrm{L}}, V^{\mathrm{sing}}$ on $\R^n$ satisfying the following properties: 
\begin{itemize}
\item[$\bullet$] Short-range part: $V^{\mathrm{S}}$ satisfies $V^{\mathrm{S}}\in L^\infty(\R^n)$ and there exists $\rho_{\mathrm{S}}>1+\frac n4$ such that  
\begin{align}
\label{assumption_A_1}
|V^{\mathrm{S}}(x)|&\lesssim \<x\>^{-\rho_{\mathrm{S}}},\quad x\in \R^n.
\end{align}
\item[$\bullet$]  Long-range part: $V^{\mathrm{L}}\in C^3(\R^n)$ and there exists $\frac n4<\rho_{\mathrm{L}}\le1 +\frac n4$ such that for any $\alpha\in \Z_+^n=(\N\cup\{0\})^n$ with $|\alpha|\le3$ there exists $C_\alpha>0$ such that
\begin{align}
\label{assumption_A_3}
|\partial_x^\alpha V^{\mathrm{L}}(x)|\le C_\alpha \<x\>^{-\rho_{\mathrm{L}}-|\alpha|},\quad x\in \R^n.
\end{align}
\item[$\bullet$]  Singular part: $V^{\mathrm{sing}}$  is relatively $H_0$-form compact, namely $|V^{\mathrm{sing}}|^{\frac12}(H_0+1)^{-\frac12}$ is a compact operator on $L^2(\R^n)$. Moreover, $V^{\mathrm{sing}}$ is compactly supported. 
\end{itemize}
\item[(A2)] The unitary group $e^{-itH}$ generated by $H=H_0+V$ satisfies the Strichartz  estimates:\begin{align}
\label{strichartz_1}
\left\|e^{-itH}f\right\|_{L^q(\R;L^r(\R^n))}&\lesssim \|f\|_{L^2(\R^n)},\\
\label{strichartz_2}
\left\|\int_0^t e^{-i(t-s)H}F(s)ds\right\|_{L^q(\R;L^r(\R^n))}&\lesssim \|F\|_{L^{\tilde q'}(\R;L^{\tilde r'}(\R^n))},
\end{align}
for any admissible pairs $(q,r)$ and $(\tilde q,\tilde r)$.
\end{itemize}
\end{assumption}
Under (A1), $V$ is infinitesimally $H_0$-form bounded: for any $\ep>0$ there exists $C_\ep>0$ such that 
$$
\int_{\R^n} |V||f|^2dx\le \ep\int_{\R^n} |\nabla f|^2dx+C_\ep\int_{\R^n} |f|^2dx,\quad f\in H^1(\R^n).
$$
Hence, by the KLMN theorem (see \cite[Theorem X.17]{ReSi}), $H=H_0+V$ generates a unique self-adjoint operator on $L^2(\R^n)$ with form domain $H^1(\R^n)$ for which we use the same symbol $H$. 

\begin{remark}
\label{remark_1} ~~ \\
(1) Thanks to the Strichartz estimate \eqref{strichartz_1}, $H$ is purely absolutely continuous and $\sigma(H)=\sigma_{\mathrm{ac}}(H)=[0,\infty)$. In particular, $H$ has neither eigenvalues nor singular continuous spectrum. \\
(2) The conditions $\rho_{\mathrm{S}}>1+\frac n4$ and $\rho_{\mathrm{L}}>\frac n4$ are due to a technical reason in our method and the optimal assumptions should be $\rho_{\mathrm{S}}>1$ and $\rho_{\mathrm{L}}>0$. However, we do not know at the moment how to remove the additional factor $\frac n4$. \\
(3) Using \eqref{strichartz_1} and \eqref{strichartz_2}, one can also obtain
\begin{align}
\label{strichartz_3}
\left\|\int_t^\infty e^{-i(t-s)H}F(s)ds\right\|_{L^q([T,\infty);L^r(\R^n))}&\lesssim \|F\|_{L^{\tilde q'}([T,\infty);L^{\tilde r'}(\R^n))}
\end{align}
which will play a crucial role in the paper. 
\end{remark}

Here we record some concrete examples of $V$: 
\begin{example}
\label{example_V}
$V$ satisfies Assumption \ref{assumption_A} if {\it one} of the following (E1)--(E4) hold: 
\begin{enumerate}
\item[(E1)] {\it Very short-range potential}. Let $n=1,2,3$ and $V=V^{\mathrm{S}}+V^{\mathrm{sing}}$, where 
\begin{itemize}
\item $V^{\mathrm{S}}\in L^\infty(\R^n)$ satisfies \eqref{assumption_A_1} with some $\rho_{\mathrm{S}}>2$. 
\item $V^{\mathrm{sing}}\equiv0$ if $n=1,2$ and $V^{\mathrm{sing}}\in L^{\frac 32}_{\mathrm{compact}}(\R^3)$ if $n=3$. 
\end{itemize}
Moreover, the negative part $V_-=\max\{0,-V\}$ of $V$ satisfies $V_-\equiv0$ if $n=1,2$ and $\|V_-\|_{L^{\frac 32}(\R^3)}<3\cdot 2^{-\frac 43}\pi^{\frac 43}$ if $n=3$. 
\item[(E2)] {\it Smooth slowly decaying potential}. Let $n=2,3$, $V\in C^\infty(\R^n)$ and there exists $\frac n4<\rho<2$ such that the following properties (H1)--(H3) hold: 
\begin{itemize}
\item[(H1)] For all $\alpha\in \Z_+^n$, there exists $C_\alpha>0$ such that
$$
|\partial_x^\alpha V(x)|\le C_\alpha \J{x} ^{-\rho-|\alpha|},\quad x\in\R^n.$$
\item[{(H2)}] There exists $C_1>0$ such that $$V(x)\ge C_1\J{x}^{-\rho},\quad x\in\R^n.$$  
\item[{(H3)}] There exists $R_0,C_2>0$ such that $$-x\cdot\nabla V(x)\ge C_2 \J{x} ^{-\rho},\quad |x|\ge R_0.$$ 
\end{itemize}
\item[(E3)] {\it Inverse power potential}. Let $n=2,3$ and $V=Z|x|^{-\rho}+\ep W$, where $\frac n4<\rho<2$ and $Z>0$. Moreover, $\ep>0$ is a sufficiently small constant and $W\in C^\infty(\R^n)$ satisfies 
$$
|\partial_x^\alpha W(x)|\le C_\alpha \<x\>^{-1-\rho-|\alpha|}.
$$
\item[(E4)] {\it Small perturbation of (E1)--(E3)}. Let $n=3$ and $V=V_1+V_2$, where $V_1$ satisfies one of (E1)--(E3) and $\norm{V_2}_{L^{3/2}(\R^3)}$ is small enough. 
\end{enumerate}
\end{example}
In particular, if  $n=2,3$, $Z>0$ and $\rho\in (n/4,2)$ then 
$$V(x)=Z\<x\>^{-\rho},\quad V(x)=Z|x|^{-\rho},$$
satisfy Assumption \ref{assumption_A}. Note that the Strichartz  estimates \eqref{strichartz_1} and \eqref{strichartz_2} were proved by \cite{RoSc,MMT,Goldberg,Beceanu} for the case (E1), and by \cite{Mizutani_JFA,Taira} for the cases (E2) and (E3), respectively. To obtain \eqref{strichartz_1} and \eqref{strichartz_2} for the case (E4), we use  Duhamel's formulas 
$$
U_H=U_{H_1}-i\Lambda_{H_1}V_2U_H,\quad \Lambda_{H}=\Lambda_{H_1}-i\Lambda_{H_1}V_2\Lambda_H
$$
where $H_1=H_0+V_1$ and we use the following notation for short: 
$$
U_H=e^{-itH},\quad \Lambda_H F=\int_0^t e^{-i(t-s)H}F(s)ds.
$$
The endpoint Strichartz estimates for $U_{H_1}$ and $\Lambda_{H_1}$ then yield
\begin{align*}
\norm{U_Hf}_{L^2L^6}&\lesssim \norm{f}_{L^2}+\norm{V_2}_{L^{3/2}}\norm{U_Hf}_{L^2L^6},\\
\norm{\Lambda_H F}_{L^2L^6}&\lesssim \norm{\Lambda_{H_1} F}_{L^2L^{6/5}}+\norm{V_2}_{L^{3/2}}\norm{\Lambda_H F}_{L^2L^6},
\end{align*}
and \eqref{strichartz_1} and \eqref{strichartz_2} thus follow if $\norm{V_2}_{L^{3/2}(\R^3)}$ is small enough.

\begin{remark}
The Strichartz  estimates are still open for the case with slowly decaying potentials in one space dimension $n=1$. This is the reason to exclude the case $n=1$ from (E2) and (E3). 
\end{remark}

On the scattering data, we impose the following. 

\begin{assumption}
\label{assumption_B}
$ \<x\>^\gamma u_+\in L^2(\R^n)$ for some $1\le \gamma\le2$ if $n=1$ and $\frac n2<\gamma<1+\frac2n$ if $n=2,3$, and there exists $c_0>0$ such that $\supp \widehat{u_+}\subset \{|\xi|\ge c_0\}$. 
\end{assumption}

Given a scattering datum $u_+$, we define the asymptotic profile $u_{\mathrm{p}}$ as follows.  We fix a cut-off function $\chi \in C^{\infty}_0 (\R^n)$ such that $0\le \chi\le1$, $\chi(x)=1$ for $|x|\le c_0/4$ and $\chi(x)=0$ for $|x|\ge c_0/3$. Using the long-range part $V^{\mathrm{L}}$ of $V$, we define a time dependent potential $V_{T_1}(t,x)$ 
by 
\begin{align}
\label{K17}
V_{T_1} (t,x) & =  V^{\mathrm{L}}(x)   \left\{1-  \chi \left(\frac{2x}{t+T_1}\right)\right\}, 
\end{align} 
where $T_1\ge1$. It is worth noting that $V_{T_1}\equiv V^{\mathrm{L}} $ for $t\ge 0$ and $|x|\ge c_0 (t+T_1)/6$. Moreover, 
\begin{align}
\label{V_ell}
|\partial_x^\alpha V_{T_1}(t,x)|\le C_\alpha \<t\>^{-\rho_{\mathrm{L}}-|\alpha|},\quad t\geq 0,\quad x\in \R^n,
\end{align}
where $C_\alpha$ is independent of $T_1$. Using this decaying condition, we will show in Section 3 that, for sufficiently large $T_1\ge1$, there exists a solution $\Psi\in C^\infty([1,\infty)\times \R^n;\R)$ to the following Hamilton-Jacobi equation 
\begin{align}
\label{K8}
- \partial _t \Psi_{} (t,x) = \frac12 |\nabla \Psi_{} (t,x)|^2  + V_{T_1} (t,x).
\end{align}
Then the asymptotic profile $u_{\mathrm{p}}$ is defined by
\begin{align}
\label{asymptotic_profile}
u_{\mathrm{p}}(t,x) := (it)^{-n/2}e^{i\Psi_{}(t,x)}\widehat{u_+}\left(x/t\right)e^{-i \nu |\widehat{u_+} (x/t)|^{2/n} \log t},
\end{align}
which also can be written as
$$
u_{\mathrm{p}}(t,x)=e^{i\Psi(t,x)}\CAL{D}(t)W(t,x),\quad W(t,x):=e^{-i\nu|\widehat{u_+}(x)|^{2/n}\log t}\widehat{u_+}(x)
$$
Here $\widehat f(\xi)=\F f(\xi)$ denotes the Fourier transform of $f$ and $\CAL{D}(t)f(x):=(it)^{-n/2}f(x/t)$. 

Now we are ready to state the main result: 
\begin{theorem}[Modified scattering]
\label{theorem_1}
Let $1\le n\le 3$, $\gamma$ be as in Assumption \ref{assumption_B} and
\begin{align}
\label{b}
\frac n4< b < \min\left\{ \frac \gamma 2,\ \rho_{\mathrm{L}},\ \rho_{\mathrm{S}} -1,1  \right\}.
\end{align}
Let $V$ and $u_+$ satisfy Assumptions \ref{assumption_A} and \ref{assumption_B}, respectively, and $\norm{\widehat {u_+}}_{L^\infty}$ be small enough. Then there exists a unique solution $u\in C(\R;L^2(\R^n))$ to \eqref{NLS1} satisfying, for any admissible pair $(q,r)$, 
\begin{align}
\label{theorem_1_1}
\norm{u(t) - u_{\mathrm{p}} (t)}_{L^2}+\norm{u-u_\mathrm{p}}_{L^q([t,\infty);L^r(\R^n))}\lesssim t^{-b},\quad t\to +\infty. 
\end{align}
\end{theorem}

\begin{remark}
The analogous result for the negative time  can be also obtained by the same proof. 
\end{remark}

As a direct consequence of Theorem \ref{theorem_1}, we also obtain the existence of modified wave operator: 
\begin{corollary}[Modified wave operator]
\label{corollary_2}
Under the same conditions in Theorem \ref{theorem_1}, there exists $\ep_0>0$ such that, for any $0<\ep\le \ep_0$, we have the modified wave operator
\begin{align*}
W^{+}_{\Psi}\, :\, u_+ \mapsto u(0),
\end{align*}
which is defined from $\SCR{F} H^{ \gamma} \cap \{ u \in L^2 ({\R}^n) \, | \,  \supp \widehat{u}\subset \{|\xi|\ge c_0\}\ \text{and}\ \norm{\widehat u}_{L^\infty}\le \ep \}$ into $L^2({\R }^n)$.
\end{corollary}

Several remarks on Theorem \ref{theorem_1} are in order. 
\begin{remark}
If $V^{\mathrm{L}}\equiv0$, then one can take $T_1=0$ and replace $\Psi$  by $\frac{|x|^2}{2t}$. Indeed, $\frac{|x|^2}{2t}$ is a solution to the free Hamilton-Jacobi equation, that is \eqref{K8} with $V_{T_1}\equiv0$. In this case, $u_{\mathrm{p}}(t)$ is reduced to \eqref{modified_profile_1} which, as mentioned above, is the well-known asymptotic profile for the NLS \eqref{NLS1} with $V\equiv0$ introduced by Ozawa \cite{Ozawa1991} for $n=1$ and Ginibre-Ozawa \cite{Ginibre_Ozawa1993} for $n=2,3$. If $\nu=0$ then
$$
u_{\mathrm{p}}(t,x)=(it)^{-n/2}e^{i\Psi_{}(t,x)}\widehat{u_+}(x/t),
$$
which is closely related with the profile employed by Yafaev \cite{Yafaev} and Derezin\'ski--G\'erard \cite{DeGe} in the linear long-range scattering theory. 
\end{remark}

\begin{remark}
As mentioned above, we assumed $\rho_{\mathrm{S}}>1+\frac n4$ for the short-range part $V^{\mathrm{S}}$ of $V$. Hence, if we consider for instance a smooth potential $V=\widetilde V^{\mathrm{S}}+V^{\mathrm{L}}$ with $V^{\mathrm{L}}$ as above  and 
$$
\widetilde V^{\mathrm{S}}(x)\gtrsim \<x\>^{-\rho},\quad -x\cdot\nabla \widetilde V^{\mathrm{S}}(x)\gtrsim \<x\>^{-\rho},\quad { \partial ^{\alpha}_x}  \widetilde V^{\mathrm{S}}(x)=O(\<x\>^{-\rho-|\alpha|}),\quad 1<\rho\le 1+\frac n4,
$$
then $u_\mathrm{p}(t)$ depends not only on $V^{\mathrm{L}}$ but also on $\widetilde V^{\mathrm{S}}$ as well, while this is not the case in the linear scattering theory where we can choose a profile depending only on $ V^{\mathrm{L}}$. This is because (due to a technical reason) we need a stronger condition \eqref{b} on the decay rate in $t$ of $\|u(t)-u_{\mathrm{p}}(t)\|_{L^2}$ than the linear case where $\|u(t)-u_{\mathrm{p}}(t)\|_{L^2}=o(1)$ is known to be sufficient. It might be possible to find a more precise  asymptotic profile which is completely independent of short-range potentials by considering the scattering theory for $u_{\mathrm{p}}(t)$. However, we do not pursue this issue here for simplicity. 
\end{remark}

In the linear scattering theory, it is well known that if $\rho_{\mathrm{L}} > 1/2$, one can choose the so-called {\em Dollard type modifier} as the asymptotic profile which is simpler than \eqref{asymptotic_profile}. As a final result in the paper, we shall show that this is also the case for \eqref{NLS1} whenever $\rho_{\mathrm{L}}>1/2+n/8$. Let 
\begin{align*}
Q(t,x) &:= \int_0^t V^{\mathrm{L}}(\tau x)  d \tau, \quad  \widetilde{V}(t, x) :=  \int_0^t \frac{1}{\tau}\left\{Q(\tau , x)+x\cdot (\nabla Q) (\tau, x)\right\}d \tau
\end{align*}
and 
\begin{align}\label{K27}
\Psi_{\mathrm{D}}(t,x):=\frac{|x|^2}{2t}-\widetilde V\left(t,\frac xt\right).
\end{align} 
We then define the Dollard type asymptotic profile $u_{\mathrm{D}}$ by
\begin{align*}
u_{\mathrm{D}} (t,x) := (it)^{-n/2}e^{i\Psi_{\mathrm{D}}(t,x)} \widehat{u_+} \left(x/t\right) e^{-i \nu |\widehat{u_+} (x/t)|^{2/n} \log t}.
\end{align*}
\begin{theorem}[Dollard type modification] 
\label{theorem_3}
Let $1\le n\le 3$ and $\gamma$ be as in Assumption \ref{assumption_B}. 
Suppose $\rho_{\mathrm{L}} > \frac{n}8 + \frac12$ and 
$
\frac n4< b < \min\left\{ \frac\gamma2,\ \rho_{\mathrm{L}}, \ 2 \rho_{\mathrm{L}} -1,\ \rho_{\mathrm{S}} -1  \right\}.
$
Let $V$ and $u_+$ satisfy Assumptions \ref{assumption_A} and \ref{assumption_B}, respectively, and $\norm{\widehat {u_+}}_{L^\infty}$ be small enough. Then there exists a unique solution $u\in C(\R;L^2(\R^n))$ to \eqref{NLS1} satisfying, for any admissible pair $(q,r)$, 
\begin{align*}
\norm{u(t) - u_{\mathrm{D}} (t)}_{L^2}+\norm{u-u_{\mathrm{D}}}_{L^q([t,\infty);L^r(\R^n))}\lesssim t^{-b},\quad t\to +\infty.
\end{align*}
\end{theorem}

\subsection{Background on the nonlinear scattering}
\label{background}
There is a vast literature on the scattering theory for the standard NLS of the form
\begin{align}
\label{NLS2}
i\partial_t u=H_0u+\nu |u|^{\alpha}u,\quad t\in \R,\ x\in \R^n,\ \nu\in \R.
\end{align}
The scenario for the asymptotic behavior of $u$ to \eqref{NLS2} as $t\to \infty$ can be (very roughly) divided into two cases.  If $\frac 2n<\alpha$ for $n=1,2$ or $\frac 2n<\alpha\le \frac{4}{n-2}$ for $n\ge3$, then $u$ can be approximated by a free solution as $t\to \pm\infty$, namely there exist scattering data $u_\pm$ such that $$\lim_{t\to\pm \infty} \|u(t)-e^{-itH_0}u_\pm\|_{L^2}=0$$ at least for smooth, localized small initial data $u(0)=u_0$ (see \cite{TsuYa}). This phenomenon is called the (small data) {\it scattering}. On the other hand, when $0<\alpha\le \frac 2n$, no nontrivial solution converges to a free solution (see \cite{Strauss,Barab}). In this sense, the case with $\alpha \le \frac 2n$ is said to be of {\it long-range type}, besides the case with $\alpha=\frac 2n$ is said to be critical. For the critical case $\alpha=\frac 2n$ and $1\le n\le3$, Ozawa \cite{Ozawa1991} and Ginibre-Ozawa \cite{Ginibre_Ozawa1993} studied the final state problem associated with \eqref{NLS2}, determining the asymptotic profile of the solutions as $t\to \infty$ for given scattering state $u_+$. Their method is based on the Dollard decomposition $$e^{-itH_0}=\CAL{M}(t)\CAL{D}(t)\F\CAL{M}(t)$$ (see \eqref{MDFM} below for details) and an analysis of the reduced ODE of the form  $$i\partial_t W(t)=t^{-1}\nu |W(t)|^{\frac2n} W(t).$$ The Dollard decomposition can be used to find out the contribution of the linear dispersion to the asymptotic profile, while the reduced ODE for the contribution of the nonlinear term. The modified scattering for the initial value problem of \eqref{NLS2} in the critical case $\alpha=\frac 2n$ and $1\le n\le3$ was studied by Hayashi-Naumkin \cite{HaNa1998}, where these two techniques also played fundamental roles. Since then, the modified scattering theory for \eqref{NLS2} or NLS with more general critical nonlinearities has been extensively studied by using the methods of the aforementioned papers (see, for instance, \cite{HaNa2006,HaWaNa,NaSu,MaMi,MMU}). We also refer to \cite{LiSo,KaPu,IfTa} for different methods.

When $V \not\equiv 0$, obtaining an explicit formula and a Dollard type decomposition for $e^{-itH}$ is  impossible in general except for a few explicit examples of $V$. Hence, generalizing the above results for \eqref{NLS2} to the case with potentials is not straightforward and the scattering theory for the NLS with linear potentials (especially in the long-range case) is much less understood compared with \eqref{NLS2}.  Nevertheless, based on the Strichartz estimates for $H=H_0+V$, there are still several works in any space dimensions if both of the nonlinear term $|u|^\alpha u$ and the linear potential $V(x)$ are of very short-range type in the sense that $\frac 4n\le \alpha\le \frac{4}{n-2}$ and $|V(x)|\lesssim \<x\>^{-2}$ at least. Moreover, the scattering theory for \eqref{NLS1} (with potentials) has been recently studied for the case when $n=1$ and $V$ is of very short-range type (see \cite{IPNa,INa,GePuRo,ChePu,Se,MaMuSe}). However, as already mentioned in the introduction, there is no previous results on the scattering theory for \eqref{NLS1} with long-range (or even short-range, but not very short-range) potentials. Moreover, if $n\ge2$ then there seems to be also no previous positive results on the scattering for \eqref{NLS1}  even if $V$ is of very short-range type. We also would emphasize that the results for the one dimensional case mentioned above used in essential ways several techniques such as the stationary scattering theory via Jost functions, which are available only in one space dimension. Hence, our results contribute to the study of  \eqref{NLS1} from a perspective of including the long-range potentials and handling two and three space dimensions. Moreover, the one of interesting point of our results is to include the Coulomb-type long-range potential; hence, we believe that the results in this paper are of great importance not only mathematically but also physically.


\subsection{Idea of the proof}
We here explain briefly the main idea of the proof of Theorem \ref{theorem_1}. For simplicity, we consider the case $V=V^{\mathrm{L}}$ and $\rho_{\mathrm L}<1$ only. 

In the linear long-range scattering theory, there are several choices of the modified free evolution. Among them, we employ a position-dependent modifier proposed by Yafaev \cite{Yafaev} (see also \cite{DeGe}). Precisely, Yafaev used the asymptotic profile of the form
$
(it)^{-n/2}e^{i\Psi_{}(t,x)}\widehat{u_+}\left(x/t\right)
$. 
 The advantage of such a modifier is that $\CAL{M}_\Psi(t):=e^{i\Psi(t,x)}$ is just a multiplication operator, so easier to treat than other known modifiers which are usually given by pseudo-differential  or Fourier integral operators (see e.g. H\"{o}rmander \cite{Hormander} or Isozaki-Kitada \cite{IsoKi}). Our asymptotic profile $u_{\mathrm{p}}$ given in \eqref{asymptotic_profile} is exactly a mixture of Yafaev's modifier $\CAL{M}_\Psi(t)$ for linear long-range scattering and Ozawa's one $W(t,x)=e^{-i\nu|\widehat{u_+}(x)|^{2/n}\log t}\widehat{u_+}(x)$ for nonlinear long-range scattering. 
 
 With the asymptotic profile $u_{\mathrm{p}}$ at hand, we reformulate \eqref{NLS1} into the following integral equation: 
\begin{align}
\label{idea_1}
u(t)=u_{\mathrm{p}}(t)+i\int_t^\infty e^{-i(t-s)H}\left\{F(u(s)) - F(u_{\mathrm{p}}(s))-(i\partial_s-H)u_{\mathrm{p}}(s)+ F(u_{\mathrm{p}}(s))\right\}ds. 
\end{align}
The final goal then is to show that the right hand side of \eqref{idea_1} is a contraction in an appropriate energy space equipped with the norm 
$$
\sup_{t\ge T} t^b\|u(t)\|_{L^2}+\sup_{t\ge T}  t^b\|u\|_{L^q([t,\infty);L^r)}
$$
with some $T$ large enough and admissible pair $(q,r)$, where $b$ is as in Theorem \ref{theorem_1}. 

The term associated with the difference $F(u) - F(u_{\mathrm{p}})$ can be dealt by the same argument based on the Strichartz estimates for $e^{-itH}$, as in the previous works for the NLS without potentials. In particular, no specific property of  $\Psi(t,x)$ will be used in this step. 

To deal with the remainder term (denoted by $\CAL{E}(t)$) associated with $-(i\partial_t-H)u_{\mathrm{p}}+ F(u_{\mathrm{p}})$, we first decompose it into a low velocity part $\CAL{E}_1(t)$ and high velocity part $\CAL{E}_2(t)$ associated with the regions $|x|\lesssim t$ and $|x|\gtrsim t$, respectively. For the low velocity part, thanks to the assumption $0\notin \supp \widehat{u_+}$, we can deal with $\CAL{E}_1(t)$ by using a simple propagation estimate based on the Dollard decomposition of $e^{-itH_0}$  (see Lemma \ref{lemma_4_2} below) and the standard nonlinear estimates
\begin{align}
\label{idea_2}
\|W(t)\|_{H^{\gamma}}\lesssim (\log t)^{\lr{\gamma}}\Gamma_{\lceil \gamma \rceil}(u_+),\quad
\|F(W(t))\|_{H^{\gamma}}\lesssim (\log t)^{\lr{\gamma}}\Gamma_{\lceil \gamma \rceil+1}(u_+)
\end{align}
with $\Gamma_{a}(u_+)=(1 + \left\| \widehat{u_+} \right\|_{H^{\gamma}}^{\frac{2a}n} )\|u\|_{H^{\gamma}}$. Again, no specific structure of $\Psi$ will be used in this step. For the high velocity part $\CAL{E}_2(t)$, we use the fact that $u_{\mathrm{p}}$ satisfies 
\begin{align*}
F(u_{\mathrm{p}}) &=\CAL{M}_\Psi(t)\CAL{D}(t)t^{-1}F(W) 
=  \CAL{M}_\Psi(t)\CAL{D}(t)i\partial_t W 
\end{align*}
to find that the main term of $(i\partial_t-H)u_{\mathrm{p}}$ and $F(u_{\mathrm{p}})$ cancel each other out, namely 
\begin{align*}
-(i\partial_t-H)u_{\mathrm{p}}+ F(u_{\mathrm{p}})
& =  e^{-itH} \left( -i\partial_t e^{itH}\CAL{M}_\Psi(t)\CAL{D}(t) W \right) +F(u_{\mathrm{p}})
\\ & = -e^{-itH} [i\partial_t,e^{itH}\CAL{M}_\Psi(t)\CAL{D}(t)   ]W,
\end{align*}
where $[\cdot,\cdot]$ denotes the commutator. If we calculate the above commutator directly, one of the unremovable term in the commutator calculation, $ \Delta \CAL{M}_\Psi(t)\CAL{D}(t) W $, can not be defined unless $W \in H^2 ({\mathbb{R} } ^n)$ at least, while to have $W \in H^2 ({\mathbb{R} } ^n)$ when $n = 3$ is difficult. Hence, we use Lemma \ref{lemma_4_3} and give some modifications to the above arguments, to change the target to handle from the above commutator to $e^{-itH}[i\partial_t,e^{itH}\CAL{M}_\Psi(t)\CAL{D}(t) \CAL{F} \CAL{M}(t) \CAL{F} ^{-1} ]W$. In the high-velocity region $|x|\gtrsim t$, this term can be written in the form
$$
\CAL{M}_\Psi(t)\left\{-\partial_t\Psi-\frac12|\nabla\Psi|^2-V_{T_1}+i\left(\nabla \Psi-\frac xt\right)\cdot\nabla +i\left(\Delta \Psi-\frac nt\right)\right\}\CAL{D}(t)\F \CAL{M}(t)\F^{-1}W,
$$
where $\CAL{M}(t)=e^{i|x|^2/(2t)}$. To obtain the time-decay of this term, we shall construct $\Psi$ in such a way that $\Psi$ satisfies the Hamilton--Jacobi equation \eqref{K8} and
$$
|\nabla \Psi-x/t|+\<t\>|\Delta \Psi-n/t|\lesssim \<t\>^{-\rho_\mathrm{L}}. 
$$
for $|x|\gtrsim t$. We follow a similar argument as in Derezi\'nski-G\'erard \cite{DG} for the construction of $\Psi$, which is based on the standard method of characteristics. This decay estimate and the bound
$$
\|\CAL{D}(t)\F \CAL{M}(t)\F^{-1}W\|_{L^2}+\<t\>\|\nabla \CAL{D}(t)\F \CAL{M}(t)\F^{-1}W\|_{L^2}\lesssim \|W(t,\cdot)\|_{H^1}
$$
as well as \eqref{idea_2} yield a desired decay estimate for the high-velocity part $\CAL{E}_2(t)$. 

\subsection{Organization of the paper} The rest of the paper is devoted to the proof of Theorems \ref{theorem_1} and \ref{theorem_3}. In Section \ref{section_2}, we introduce the integral equation we will solve more precisely. We also recall the Dollard decomposition of $e^{-itH_0}$ and some nonlinear estimates (Lemma \ref{lemma_nonlinear}) in Section \ref{section_2}. The construction of $\Psi$ is given in Section 3. We prove several necessary  energy estimates for $\CAL{E}(t)$ in Section 4. The final step of the proof of Theorem \ref{theorem_1} is given in Section 5. The proof of Theorem \ref{theorem_3} is given in Section 6. For reader's convenience, we give the proof of existence of global $L^2$-solution in  Appendix \ref{appendix_proposition_GWP} and the proof of the nonlinear estimates in Appendix \ref{appendix_lemma_nonlinear}, respectively. 


\section{Preliminaries}
\label{section_2}
Throughout the paper, we assume Assumptions \ref{assumption_A} and \ref{assumption_B}.

\subsection{Notation}
\label{subsection_notation}
We first introduce some notations.  
\begin{itemize}
\item $\lceil \gamma\rceil=\min\{m\in \Z\ |\ m\ge\gamma\}$ denotes the smallest integer greater than or equal to $\gamma$. 
\item For positive constants $A$ and $B$, $A\lesssim B$ means $A\le CB$ for some non-important constant $C$. We define $A\gtrsim B$ similarly. $A\sim B$ means $A\lesssim B$ and $A\gtrsim B$. 
\item Let $L^p=L^p(\R^n)$ be the Lebesgue space and $\norm{u}_{L^p}=\norm{u}_{L^p(\R^n)}$. 
\item $H^s=H^s(\R^n)$ denotes the $L^2$-Sobolev space of order $s$ with norm $\norm{f}_{H^s}=\norm{\<D\>^sf}_{L^2}$. 
\end{itemize}

\subsection{Dollard decomposition}

Here we recall the Dollard decomposition  \eqref{MDFM} (also often called the MDFM decomposition) of the free propagator $e^{-itH_0}$. Let 
\begin{align}
\label{MD}
\CAL{M}(t)f(x):=e^{\frac{i|x|^2}{2t}}f(x),\quad \CAL{D}(t) f(x) := (it)^{-n/2} f(x/t).
\end{align}
and $\F$ be the Fourier transform. Note that all of them are unitary on $L^2(\R^n)$. 
Then we have
\begin{align}
\label{MDFM}
e^{-itH_0}f(x)=(2\pi it)^{-n/2}\int_{\R^n}e^{\frac{i|x-y|^2}{2t}}f(y)dy=\CAL{M}(t)\CAL{D}(t)\SCR{F}\CAL{M}(t)f(x)
\end{align}
by a direct calculation. Here we record two basic estimates for $\CAL{M}(t)$ and $\CAL{D}(t)$: 
\begin{align}
\label{Dollard_1}
\|\CAL{D}(t)f\|_{L^r}=|t|^{-n(1/2-1/r)}\|f\|_{L^r},\quad \|(\CAL{M}(t)-1)f\|_{L^r}\lesssim |t|^{-\delta/2}\||x|^{\delta} f\|_{L^r}
\end{align}
for all $1\le r\le\infty$ and $0\le \delta \le 2$. The latter estimate follows from the bound $|e^{\frac{i|x|^2}{2t}}-1|\lesssim (|x|^2/t)^{\delta/2}$.


\subsection{Integral equation}
\label{subsection_integral_equation}
Next we introduce the integral equation associated with our problem. Using the above notation, we observe from \eqref{MDFM} that
$$
\CAL{D}(t)=\CAL{D}(t)\SCR{F}\left\{1-\CAL{M(t)}\right\}\SCR{F}^{-1}+\CAL{M}(-t)e^{-itH_0}\SCR{F}^{-1}. 
$$
Set $\CAL{M}_\Psi(t)f(x) :=e^{i\Psi(t,x)}f (x) $ and $W(t,x):= e^{-i \nu \left| \widehat{u_+}(x) \right|^{2/n} \log t } \widehat{u_+}(x)$. Note that $W$ satisfies
\begin{align}
\label{integral_1}
i \partial _t W(t,x) = \frac{1}{t}F(W(t,x)),\quad t\neq0,\ x\in \R^n. 
\end{align}
Let $\chi\in C_0^\infty(\R^n)$ be as in Section 1 (see \eqref{K17}) and set $\chi_t(x)=\chi(x/t)$. Define the operators $U_{\Psi} (t)$, $\CAL{U}_j(t)$ for $j=1,2,3$ by
\begin{align*}
U_{\Psi} (t) &=  \CAL{M}_\Psi(t) \CAL{M}(-t) e^{-i t H_0},\\
\CAL{U}_1(t) & =    \CAL{M}_\Psi(t) \CAL{D}(t) \SCR{F}\left\{ 1- \CAL{M}(t)  \right\}\SCR{F}^{-1},\\
\CAL{U}_2(t) &= \chi(x/t)U_{\Psi} (t)\SCR{F}^{-1},\\ 
\CAL{U}_3(t)&=\left\{1-\chi(x/t)\right\}U_{\Psi} (t)\SCR{F}^{-1}.
\end{align*}
Then we can write 
\begin{align}
\label{integral_2}
\CAL{M}_\Psi(t)\CAL{D}(t)=\CAL{M}_\Psi(t) \CAL{D}(t) \SCR{F}\left\{ 1- \CAL{M}(t)  \right\}\SCR{F}^{-1}+\{\chi_t+(1-\chi_t)\} U_\Psi(t)\SCR{F}^{-1}
=\sum_{j=1}^3\CAL{U}_j(t)
\end{align}
so that
$$u_{\mathrm{p}}(t,x)=\CAL{M}_\Psi(t)\CAL{D}(t) W(t,x)=\sum_{j=1}^3\CAL{U}_j(t)W(t,x).$$
It will be seen in Lemma \ref{lemma_4_3} in Section 4 that 
\begin{align}
\label{integral_3}
\lim_{t\to\infty}\|\CAL{U}_j(t)W(t)\|_{L^2}=0,\quad j=1,2.
\end{align}
Note that \eqref{integral_3} can fail for $j=3$ in general. Now we assume for a while that $u$ is a smooth solution to \eqref{NLS1} satisfying $\|u-u_{\mathrm{p}}\|_{L^2}\to 0$ as $t\to +\infty$. Then \eqref{NLS1}
leads the Duhamel formula: 
\begin{align}
\label{integral_4}
u(t)-u_{\mathrm{p}}(t)=i\int_t^\infty e^{-i(t-s)H}\left\{F(u(s)) - F(u_{\mathrm{p}}(s))-(i\partial_s-H)u_{\mathrm{p}}(s)+ F(u_{\mathrm{p}}(s))\right\}ds. 
\end{align}
For the term $(i\partial_s-H)u_{\mathrm{p}}(s)$, since 
$$
e^{isH}(i\partial_s-H)u_{\mathrm{p}}(s)=i\partial_s(e^{isH}u_{\mathrm{p}}(s))=\sum_{j=1}^3i\partial_s (e^{isH}\CAL{U}_j(s)W(s)),
$$
we know by \eqref{integral_3} that
\begin{align}
\nonumber
&-i\int_t^\infty e^{-i(t-s)H}(i\partial_s-H)u_{\mathrm{p}}(s)ds\\
\label{integral_5}
&=-\CAL{U}_1(t)W(t)-\CAL{U}_2(t)W(t)-i\int_t^\infty e^{-itH}i\partial_s(e^{isH}\CAL{U}_3(s)W(s))ds.
\end{align}
Here \eqref{integral_1} yields that the last term of the RHS is written in the form
\begin{align}
\nonumber
&-i\int_t^\infty e^{-itH}i\partial_s(e^{isH}\CAL{U}_3(s)W(s))ds\\
\label{integral_6}
&=-i\int_t^\infty e^{-i(t-s)H}\CAL{C}(s)W(s)ds-i\int_t^\infty e^{-i(t-s)H}\CAL{U}_3(s)\frac{F(W(s))}{s}ds,
\end{align}
where $\CAL{C}(s)$ denotes an extension of the commutator
$$
e^{-isH}[i\partial_s,e^{isH}\CAL{U}_3(s)]=e^{-isH}i\partial_se^{isH}\CAL{U}_3(s)-i\CAL{U}_3(s)\partial_s.
$$
In Lemma \ref{lemma_4_4} below, we will give the precise definition of $\CAL{C}(s)$ and show that $\CAL{C}(s)W(s)$ belongs to $L^1([T,\infty);L^2(\R^n))$, whenever $\widehat u_+\in H^\gamma$ with $\gamma>\frac n2$ and $T$ is large enough. For the last term $F(u_{\mathrm{p}}(t))$ in \eqref{integral_4}, we find by \eqref{MD} 
and \eqref{integral_2} that
\begin{align}
\label{integral_7}
F(u_{\mathrm{p}}(t))=\CAL{M}_\Psi(t)\CAL{D}{(t)}\frac{F(W(t))}{t}=\sum_{j=1}^3\CAL{U}_j(t)\frac{F(W(t))}{t},
\end{align}
where the term $i\int_t^\infty e^{-i(t-s)H}\CAL{U}_3(s)\frac{F(W(s))}{s}ds$ and the last term of \eqref{integral_6} cancel each other out. 

With the above calculations at hand, we now introduce the integral equation we will solve. Set
\begin{align}
\nonumber
\CAL{E}(t) &= \CAL{E}_1 (t) + \CAL{E}_2(t),\\
\nonumber
\CAL{E}_1 (t) &= - \left(\CAL{U}_1(t)+\CAL{U}_2(t) \right)W(t)
 + i \int _t^{\infty} e^{-i(t-s)H } \left(\CAL{U}_1(s)+\CAL{U}_2(s)\right) \frac{F(W(s))}{s} ds,  \\ 
\nonumber
\CAL{E}_2 (t) & = -i \int _t^{\infty} e^{-i(t-s)H}\CAL{C}(s)W(s) {ds},\\
 \label{K3}
\CAL{K}[u](t)&= i \int_t^{\infty} e^{-i (t-s) H} \left( F(u(s)) -F(u_{\mathrm{p}} (s)) \right) ds.
\end{align}
Then \eqref{integral_4}--\eqref{integral_6} lead the following integral equation for $u(t)$: 
\begin{align}\label{K21}
u(t) =u_{\mathrm{p}} (t)+ \CAL{K}[u](t) + \CAL{E}(t).
\end{align}
We define the energy space $X(T,R,q,r)$ by
$$
X(T,R,q,r):=\{u\in C([T,\infty);L^2(\R^n) )\ |\ \|u-u_{\mathrm{p}}\|_{X_{T,q,r}}\le R\}
$$
equipped with a distance function $d(u,v)=\|u-v\|_{X_{T,q,r}}$, where
$$
\|u\|_{X_{T,q,r}} := \sup_{t \geq T} t^b \| u (t) \|_{L^2(\R^n)} + \sup_{t\geq T} t^b \| u\|_{L^q([T,\infty);L^r(\R^n))}.
$$
In the following sections, we will construct the solution $u$ to \eqref{K21} satisfying \eqref{theorem_1_1} by showing that the map $u\mapsto u_{\mathrm{p}}+ \CAL{K}[u] + \CAL{E}$ is a contraction in $X(T,R,q,r)$ with an appropriate $T,q,r$ and any $R$. Although \eqref{K21} is different from the usual integral equation for \eqref{NLS1}, the next lemma shows that the solution to \eqref{K21} is in fact a solution to the standard integral equation associated with \eqref{NLS1}. 

\begin{lemma}
\label{lemma_integral_1}
Suppose $u\in C([t_0,\infty);L^2(\R^n))\cap L^q([t_0,\infty);L^r(\R^n))$ for any admissible pair $(q,r)$ and $u$ is a solution to \eqref{K21} with some $t_0\in \R$. 
Then
\begin{align}
\label{lemma_integral_1_1}
u(t) = e^{-i(t-t_0)H} u(t_0) -i \int_{t_0}^{t} e^{-i(t-s) H} F(u(s)) ds, \quad t \geq t_0.
\end{align}
\end{lemma}

\begin{proof}
It is enough to check that the right hand sides of \eqref{K21} and \eqref{lemma_integral_1_1} coincide with each other, which, after changing the variable $t\to t+t_0$, is equivalent to that the following holds for $t\ge0$: 
\begin{align}\nonumber 
0&=u_{\mathrm{p}} (t+t_0) - i \int_{t+t_0}^{\infty} e^{-i (t+t_0-s) H} F(u_{\mathrm{p}} (s)) ds+\CAL{E}(t+t_0)\\&\quad -e^{-itH} u(t_0)+i \int_{t_0}^{\infty} e^{-i (t+t_0-s) H} F(u(s))ds. 
\label{K-4/26}
\end{align}
By \eqref{integral_2}, \eqref{integral_3}, \eqref{integral_6}, \eqref{integral_7} and \eqref{K3}, the sum of the first three terms of the RHS of \eqref{K-4/26} is equal to
\begin{align*}
&u_{\mathrm{p}} (t+t_0) - i \int_{t+t_0}^{\infty} e^{-i (t+t_0-s) H} F(u_{\mathrm{p}} (s)) ds+\CAL{E}(t+t_0)\\
&=\CAL{U}_3(t+t_0)W(t+t_0) - i e^{-i(t+t_0)H}\int_{t+t_0}^{\infty} e^{isH}\left\{\CAL{U}_3(s)\frac{F(W(s))}{s}+\CAL{C}(s) W(s) \right\}{ds} \\
&= \CAL{U}_3(t+t_0)W(t+t_0)-ie^{-i(t+t_0)H}\int_{t+t_0}^\infty i\partial_s(e^{isH}\CAL{U}_3(s)W(s))ds.
\end{align*}
By using, in addition, \eqref{K21} with $t=t_0$ we similarly find that the sum of the last two terms of the RHS of \eqref{K-4/26} is equal to
\begin{align*}
&-e^{-itH} u(t_0)+i \int_{t_0}^{\infty} e^{-i (t+t_0-s) H} F(u(s))ds\\
&=-e^{-itH}u_{\mathrm{p}}(t_0)+i\int_{t_0}^\infty e^{-i(t+t_0-s)H}F(u_{\mathrm{p}}(s))ds-e^{-itH}\CAL{E}(t_0)\\
&=-e^{-itH}\CAL{U}_3(t_0)W(t_0)+ie^{-i(t+t_0)H} \int_{t_0}^\infty i\partial_s(e^{isH}\CAL{U}_3(s)W(s))ds.
\end{align*}
Hence the RHS of \eqref{K-4/26} is equal to
$$
\CAL{U}_3(t+t_0)W(t+t_0)-e^{-itH}\CAL{U}_3(t_0)W(t_0)+ie^{-i(t+t_0)H}\int_{t_0}^{t+t_0} i\partial_s(e^{isH}\CAL{U}_3(s)W(s))ds
$$
which vanishes identically. 
\end{proof}

\subsection{Nonlinear estimates}
We end this preliminary section to record the following nonlinear estimates, which play an important role in analyzing the $L^qL^r$-norms of $\CAL{K}[u](t)$.
\begin{lemma}
\label{lemma_nonlinear}
Let $n\ge1$, $\frac n2<\gamma<\min\left\{2,1+\frac2n\right\}$ and $\lambda\ge1$. Then
\begin{align}
\label{lemma_nonlinear_1}
\|e^{i\lambda |u|^{\frac2n}}u\|_{H^\gamma}&\lesssim \big(1+\lambda^{\lceil \gamma \rceil}\|u\|_{H^\gamma}^{\frac{2\lceil \gamma \rceil}{n}}\big)\|u\|_{H^\gamma}, \\
\label{lemma_nonlinear_2}
\|e^{i\lambda |u|^{\frac2n}} |u|^{\frac 2n}u\|_{H^\gamma}&\lesssim \big(1+\lambda^{\lceil \gamma \rceil}\norm{u}_{H^\gamma}^{\frac{2\lceil \gamma \rceil}{n}}\big)
\norm{u}_{L^{\infty}}^{\frac{2}{n}} \| u\|_{H^\gamma}.
\end{align}
\end{lemma}

\begin{proof}
The lemma is a special case of \cite[Lemma 4]{HaNa2006}. However, we provide a complete proof in Appendix \ref{appendix_lemma_nonlinear} for the sake of completeness. 
\end{proof}



\section{The phase function $\Psi$} 
\label{section_phase}
In this section, we construct the phase function $\Psi$ of the asymptotic profile $u_{\mathrm{p}}$ and prove its asymptotic property, which will be used in the study of the term $\CAL{E}_2(t)$. 
\begin{proposition}\label{proposition_3_1}
For sufficiently large $T_1\ge1$, there exists a solution $\Psi\in C^1([1,\infty)\times\R^n;\R)$ to \eqref{K8} such that $\Psi$ is $C^2$ with respect to $x \in \R^n$ and, for all $t\ge1$, 
\begin{align} \label{proposition_3_1_1}
\left\| \nabla \Psi_{} (t,x)-\frac xt
\right\|_{L^\infty(\R^n)} &\lesssim \max\{t^{- \rho_{\mathrm{L}}},t^{-1}\log t\},\\
\label{proposition_3_1_2}
\left\| \Delta\Psi_{} (t,x)-\frac nt
\right\|_{L^\infty(\R^n)} &\lesssim  \max\{t^{-1- \rho_{\mathrm{L}}},t^{-2}\log t\}. 
\end{align}
\end{proposition}

A similar statement as this proposition was previously obtained by Derezi\'{n}ski-G\'{e}rard \cite{DeGe} (see also Yafaev \cite{Yafaev}). They proved $\partial _x^{\alpha} ( \Psi (t,x) - \frac{|x|^2}{2t})=o(1)$ as $t\to +\infty$ if $|\alpha|=1,2$ under somewhat weaker assumption than (A1) in Assumption \ref{assumption_A}, which was sufficient for the linear long-range scattering theory. However, the concrete decay estimates \eqref{proposition_3_1_1} and \eqref{proposition_3_1_2} play an important role in the present nonlinear problem since it affects  the decay rate of $\CAL{E}(t)$ and thus will be needed to show that the nonlinear map $u\mapsto u_{\mathrm{p}}+\CAL{K}[u]+\CAL{E}$ is a contraction in the space $X(T,R,q,r)$.

The proof of this proposition follows closely the argument by Derezi\'{n}ski-G\'{e}rard \cite[Sections 1.5, 1.8 and A.3]{DG}. We may assume without loss of generality 
$$
0<\rho_{\mathrm L}\le 1
$$
in this section. For the case $1<\rho_{\mathrm L}\le 1+n/4$, it is enough to replace $\rho_{\mathrm L}$ by $1$ in the following argument. Consider the Hamilton equation 
\begin{align}\label{K5}
\frac{d}{dt} X(t ,\xi) = \Xi (t,\xi), \quad 
\frac{d}{dt} \Xi(t, \xi) = - (\nabla_x V_{T_1}) (t, X (t, \xi)),\quad t\ge0,
\end{align}
with the given initial position at $t=0$ and final momentum at $t\to\infty$: 
\begin{align}
\label{Hamilton_data}
X(0,\xi)=0,\quad \lim_{t\to\infty}\Xi(t,\xi)=\xi.
\end{align}
We often suppress the variables $\xi$ for short if there is no confusion. 
\begin{lemma}
\label{lemma_3_2}
For any $\xi\in \R^n$ and sufficiently large $T_1\ge1$, there exists a unique solution $(X,\Xi)\in C^1([0,\infty)\times\R^{n};\R^{2n})$ to \eqref{K5}--\eqref{Hamilton_data} such that, for $\alpha\in \Z^n_+$ satisfying $|\alpha|\le1$, 
\begin{align}
\label{lemma_3_2_1}
|\partial_\xi^\alpha (X(t, \xi)  - t\xi) |&\lesssim 
\begin{cases}
\J{t} (t+T_1)^{-\rho_{\mathrm L}}&\text{if}\ \rho_{\mathrm L}<1,\\ \log(tT_1^{-1}+2)&\text{if}\ \rho_{\mathrm L}=1,
\end{cases}\\
\label{lemma_3_2_2}
|\partial_\xi^\alpha(\Xi(t, \xi) -\xi)|&\lesssim (t+T_1)^{-\rho_{\mathrm L}}.
\end{align}
\end{lemma}

\begin{proof}
The system \eqref{K5}--\eqref{Hamilton_data} is rewritten as
\begin{align}
\label{lemma_3_2_proof_1}
X(t,\xi)&=t\xi+\int_0^t\int_r^\infty (\nabla_xV_{T_1})(s,X(s,\xi))dsdr
=t\xi+\int_0^\infty \min\{t,s\}(\nabla_xV_{T_1})(s,X(s,\xi))ds,\\
\label{lemma_3_2_proof_2}
\Xi(t,\xi)&=\xi+\int_t^\infty (\nabla_xV_{T_1})(s,X(s,\xi))ds. 
\end{align}
Setting for short
$${\displaystyle \theta_{\mathrm{L}} (t) = \begin{cases}
{ \J{t} }(t+T_1)^{-\rho_{\mathrm L}}&\text{if}\ \rho_{\mathrm L}<1,\\ \log(tT_1^{-1}+2)&\text{if}\ \rho_{\mathrm L}=1,
\end{cases}}$$ 
we define the complete metric space $(\mathcal Z_M, d)$ by
\begin{align*}
\mathcal Z_{M}^{} &= \left\{ Z\in C^{1}([0,\infty) ; \R ^n) \, \middle| \, \left\| Z \right\|_{\mathcal Z} :=\sup_{\xi\in \R^n}\sup_{t \in [0, \infty) } \frac{|   Z(t, \xi) | + | \nabla_{\xi} Z(t, \xi) | }{\theta_{\mathrm{L}} (t)}  \leq M \right\}, 
 \\
 d (\phi, \psi) &=  \left\| \phi - \psi  \right\|_{\mathcal Z}
 \end{align*} 
with some constant $M>0$ specified later, and set
 \begin{align*} 
   \Phi [Z](t,\xi) = \int_0^\infty \min\{t,s\}\nabla_xV_{T_1}(s,Z(s,\xi)+s\xi)ds,
\end{align*}
where, by \eqref{V_ell}, $\nabla_xV_{T_1}(t,x)$ satisfies 
\begin{align}
\label{V}
|\partial_x^\alpha \nabla_xV_{T_1}(t,x)|\lesssim (t+T_1)^{-1-\rho_{\mathrm L}-|\alpha|},\quad |\alpha|\le 1,
\end{align}
uniformly in $t>0$, $x\in \R^n$ and $T_1\ge1$. We shall show $\Phi$ is a contraction on $\mathcal Z_M$. 

Suppose first $Z\in \mathcal Z_{M}$. It follows from \eqref{V} that 
\begin{align*}
 \left|\int_0^\infty \min \{ t,s \} \nabla_x V_{T_1}(s,x)ds\right|&\lesssim \int_0^\infty \min\{t,s\}(s+T_1)^{-1-\rho_{\mathrm L}}ds\\
 &\lesssim\int_0^t\left\{(s+T_1)^{-\rho_{\mathrm L}}-T_1(s+T_1)^{-1-\rho_{\mathrm L}}\right\}ds+t\int_t^\infty (s+T_1)^{-1-\rho_{\mathrm L}}ds\\
&\lesssim \theta_{\mathrm{L}} (t)
\end{align*}
uniformly in $t\ge0$, $x\in \R^n$ and $T_1\ge1$. Hence 
\begin{align}
\label{lemma_3_2_proof_3}
\ds C_1:=\sup_{\xi\in \R^n}\sup _{t \in [0,\infty)}\theta_{\mathrm{L}} (t)^{-1} |\Phi[Z] |<\infty.
\end{align}
Note that $C_1$ is independent of $T_1$. Similarly, we obtain
\begin{align}
\nonumber
|\nabla_\xi\Phi[Z](t,\xi)|
&\lesssim \int_0^\infty \min\{t,s\}(s+T_1)^{-2-\rho_{\mathrm L}}(|\nabla_\xi Z(s,\xi)|+s)ds\\
\nonumber
&\lesssim M\int_0^\infty \min\{t,s\}(s+T_1)^{-2-\rho_{\mathrm L}}\theta_{\mathrm L}(s)ds+\int_0^\infty \min\{t,s\}(s+T_1)^{-2-\rho_{\mathrm L}}sds\\
\nonumber
&\lesssim M\int_0^\infty \min\{t,s\}(s+T_1)^{-1-\frac{3\rho_{\mathrm L}}{2}}ds+\theta_{\mathrm L}(t)\\
\label{lemma_3_2_proof_4}
&\le C_2(T_1^{-\frac{\rho_{\mathrm L}}{2}}M+1)\theta_{\mathrm L}(t)
\end{align}
with some $C_2>0$ independent of $t,\xi,M,T_1$. Let $C_2T_1^{-\rho_{\mathrm L}/2}\le 1/4$ and $M=\max\{2C_1,4C_2\}$. Then the RHS of \eqref{lemma_3_2_proof_4} is dominated by $M\theta_{\mathrm L}(t)/2$, and thus $\Phi[Z]\in \mathcal Z_M$ by \eqref{lemma_3_2_proof_3} and \eqref{lemma_3_2_proof_4}. Next, if $Z_1,Z_2\in \mathcal Z_M$, then we obtain by the fundamental theorem of calculus and \eqref{V}, \begin{align}
\nonumber
\left| (\nabla V_{T_1} ) (t, Z_1+t\xi) -(\nabla V_{T_1} ) (t, Z_2+t\xi)  \right| 
&\lesssim (t + T_1)^{-2-\rho_{\mathrm{L}}} | Z_1 - Z_2| \\
\label{lemma_3_2_proof_5}
&\lesssim  M(t+T_1)^{-2 - \rho_{\mathrm{L}}}\theta_{\mathrm L}(t)  \left\| Z_1- Z_2 \right\|_{\mathcal Z},
\end{align} 
and similarly 
\begin{align}
\nonumber
&\left| \nabla_\xi \{(\nabla_xV_{T_1} ) (t, Z_1+t\xi) -(\nabla_xV_{T_1} ) (t, Z_2+t\xi)\} \right|\\
\nonumber
&\lesssim (t + T_1)^{-2-\rho_{\mathrm{L}}} | \nabla_\xi (Z_1 - Z_2)| +  (t + T_1)^{-3-\rho_{\mathrm{L}}} ( | \nabla_\xi Z_1| +  | \nabla_\xi Z_2| )     | Z_1- Z_2 | \\
\label{lemma_3_2_proof_6}
& \lesssim  M(t+T_1)^{-2 - \rho_{\mathrm{L}}}\theta_{\mathrm L}(t)  \left\| Z_1- Z_2 \right\|_{\mathcal Z}. 
 \end{align} 
Since $M$ is independent of $T_1$ and
$$
\int_0^\infty \min\{t,s\}(s+T_1)^{-2-\rho_{\mathrm L}}\theta_{\mathrm L}(s)ds\lesssim T_1^{-\rho_{\mathrm L}/2}\theta_{\mathrm L}(t),
$$
\eqref{lemma_3_2_proof_5} and \eqref{lemma_3_2_proof_6} yield
$$
\|\Phi[Z_1]-\Phi[Z_2]\|_{\mathcal Z}<\frac12\left\| Z_1- Z_2 \right\|_{\mathcal Z},
$$
provided that $T_1$ is large enough. $\Phi$ thus is a contraction on $\mathcal Z_M$ and there exists a unique solution $Z\in \mathcal Z_M$ to the equation $Z=\Phi[Z]$. 

Now $X:=Z+t\xi\in C^1([0,\infty)\times \R^n)$ is a unique solution to \eqref{lemma_3_2_proof_1} satisfying \eqref{lemma_3_2_1}. Finally, $\Xi$ defined by \eqref{lemma_3_2_proof_2} belongs to $C^1([0,\infty)\times\R^{n})$ and satisfies
$$
|\partial_\xi^\alpha (\Xi(t,\xi)-\xi)|\lesssim \int_t^\infty (s+T_1)^{-1-\rho_{\mathrm L}-|\alpha|}|\nabla_\xi X(s,\xi)|^{|\alpha|}ds\lesssim (t+T_1)^{-\rho_{\mathrm L}}
$$
for $|\alpha|\le1$, where we have used the bound  $|\nabla_\xi X(t,\xi)|\lesssim \<t\>$ which follows from \eqref{lemma_3_2_1}. 
 \end{proof}

\begin{proof}[Proof of Proposition \ref{proposition_3_1}]
The proof consists of three steps. At first, Lemma \ref{lemma_3_2} implies
\begin{align}
\label{proposition_3_1_proof_1}
|\nabla_\xi \Xi(t,\xi)-I|\lesssim (t+T_1)^{-\rho_{\mathrm{L}}}
\end{align}
on $[0,\infty)\times\R^n$. By virtue of Hadamard's inverse mapping theorem, there exists $T_1>0$ such that, for all $t\ge0$,  the map  $\R^n\in \xi\mapsto \Xi(t,\xi)$ is diffeomorphic and  has the  inverse $\xi\mapsto \eta(t,\xi)$ which satisfies the same estimate as \eqref{lemma_3_2_2}, namely 
\begin{align}
\label{proposition_3_1_proof_2}
|\partial_\xi^\alpha (\eta(t,\xi)-\xi)|\lesssim(t+T_1)^{-\rho_{\mathrm L}},\quad t\ge 0,\ \xi\in \R^n. 
\end{align}
Indeed, if $|\alpha|=0$ then
$$
|\eta(t,\xi)-\xi|=|\eta(t,\xi)-\Xi(t,\eta(t,\xi))|\lesssim (t+T_1)^{-\rho_{\mathrm L}}.
$$
Moreover, taking $T_1$ large if necessary so that $|\nabla_\xi \Xi(t)|\ge 1/2$, the inverse matrix $[\nabla_\xi \Xi(t)]^{-1}$ exists and satisfies $|\nabla_\xi \Xi(t)^{-1}-I|\lesssim (t+T_1)^{-\rho_{\mathrm L}}$ by \eqref{proposition_3_1_proof_1}. Thus, we obtain for $|\alpha|=1$ 
$$
|\partial_\xi^\alpha (\eta(t,\xi)-\xi)|=|[ (\nabla_\xi \Xi )(t,\eta)]^{-1}\partial_\xi^\alpha \xi-\partial_\xi^\alpha \xi|\lesssim (t+T_1)^{-\rho_{\mathrm L}}. 
$$

Next, we define $S(t,\xi)=\varphi(t,\eta(t,\xi))$, where 
\begin{align*}
\varphi(t, \xi) = \int_0^t\left(\frac12|\Xi(\tau,\xi)|^2+V_{T_1}(\tau,X(\tau,\xi))+X(\tau,\xi)\cdot (\partial_t \Xi)(\tau,\xi)\right)d\tau.
\end{align*} 
We shall show that $S$ satisfies
\begin{empheq}[left = {~~ \empheqlbrace \,}]{alignat = 2}
\label{proposition_3_1_proof_3}
  \partial _t S(t,\xi) &= \frac12 |\xi| ^2 +V_{T_1}(t, \nabla _{\xi} S(t,\xi)) ,  \\ 
\label{proposition_3_1_proof_4}
 \nabla_\xi S(t,\xi) &=X(t,\eta(t,\xi)) . 
\end{empheq}
It follows from \eqref{K5} that
\begin{align*}
\partial_t\partial_{\xi_j} \varphi(t)
&=\Xi(t) \cdot \partial_{\xi_j} \Xi(t)+(\nabla_x V_{T_1})(t,X(t)) \cdot \partial_{\xi_j} X(t)+\partial_{\xi_j}X(t)\cdot\partial_t \Xi(t) +X(t)\cdot \partial_t \partial_{\xi_j}\Xi(t)\\
&=\partial_t X(t)\cdot \partial_{\xi_j} \Xi(t)+X(t)\cdot \partial_t \partial_{\xi_j}\Xi(t)\\
&=\partial_t [X(t)\cdot \partial_{\xi_j} \Xi(t)],
\end{align*}
which implies $\partial_{\xi_j} \varphi(t)=X(t)\cdot \partial_{\xi_j} \Xi(t)$ since $\nabla_\xi \varphi(0)=0$. Since $\Xi(t,\eta(t,\xi))=\xi$ and hence
$$
\delta_{jk}=\partial_{\xi_j}[\Xi_k(t,\eta)]=\sum_{\ell=1}^n(\partial_{\xi_\ell}\Xi_k)(t,\eta)\partial_{\xi_j}\eta_\ell
$$
we have for $j=1,...,n$,
\begin{align*}
\partial_{\xi_j}S(t)
&=\partial_{\xi_j}[\varphi(t,\eta(t))] 
\\
&=\sum_{k,\ell=1}^nX_k(t,\eta(t))(\partial_{\xi_\ell} \Xi_k)(t,\eta(t))\partial_{\xi_j}\eta_\ell=X_j(t,\eta(t)),
\end{align*}
which implies \eqref{proposition_3_1_proof_4}. Using \eqref{proposition_3_1_proof_4}, we obtain that
$$\partial_t [S(t,\Xi(t))]=(\partial_tS)(t,\Xi(t))+X(t,\eta(t,\Xi(t))\cdot \partial_t \Xi(t)=X(t)\cdot \partial_t \Xi(t)$$
and that, since $S(t,\Xi(t))=\varphi(t,\xi)$, 
\begin{align*}
\partial_t [S(t,\Xi(t))]
&=\frac12 |\Xi(t)| ^2 +V_{T_1}(\tau,X(t,\xi))+X(t,\xi)\cdot (\partial_t \Xi)(t,\xi)\\
&=\frac12 |\Xi(t)| ^2 +V_{T_1}(\tau,\nabla_\xi S(t,\Xi(t)))+X(t,\xi)\cdot (\partial_t \Xi)(t,\xi).
\end{align*}
These two formula imply \eqref{proposition_3_1_proof_3} with $\xi$ replaced by $\Xi(t,\xi)$. Plugging $\xi=\eta(t)$ into the obtained equation, we arrive at \eqref{proposition_3_1_proof_3}. 

Finally, we shall construct $\Psi(t,x)$ by using $S(t,\xi)$. By Lemma \ref{lemma_3_2} and \eqref{proposition_3_1_proof_2}, 
\begin{align}
\nonumber
|\partial_\xi^\alpha (\nabla_\xi S(t,\xi) - t\xi)|
&\le |\partial_\xi^\alpha(X(t,\eta(t)) - t\eta(t))|+t|\partial_\xi^\alpha(\eta(t) - \xi)|\\
\label{proposition_3_1_proof_5}
&\lesssim \begin{cases}
t(t+T_1)^{-\rho_{\mathrm L}}&\text{if}\ \rho_{\mathrm L}<1,\\ \log(tT_1^{-1}+1)&\text{if}\ \rho_{\mathrm L}=1,
\end{cases}
\end{align}
for $|\alpha|=0,1$. Hence, the map $\R^n\ni \xi\mapsto t^{-1}\nabla S(t,\xi)$ is diffeomorphic  for sufficient large $T_1$ and all $t\ge0$ so that its inverse $\widetilde \Theta(t,\xi)$ satisfies $(\nabla_\xi S)(t,\widetilde \Theta(t,\xi))=t\xi$. Setting $$\Theta (t,x)=\widetilde \Theta(t,x/t),\quad x\in \R^n,\ t>0,$$ we find
\begin{align}
\label{proposition_3_1_proof_6}
(\nabla_\xi S)(t, \Theta(t,x))=x. 
\end{align}
Plugging $\xi=\Theta(t,x)$ into \eqref{proposition_3_1_proof_5} shows
\begin{align}
\label{proposition_3_1_proof_7}
\left|\Theta(t,x)-t^{-1}x\right|\lesssim 
\begin{cases}
(t+T_1)^{-\rho_{\mathrm L}}
&\text{if}\ \rho_{\mathrm L}<1,\\ t^{-1}\log(tT_1^{-1}+1)&\text{if}\ \rho_{\mathrm L}=1,
\end{cases}
\end{align}
Moreover, differentiating the both sides of \eqref{proposition_3_1_proof_6} in $x$ implies
$$
(\nabla_\xi^2S)(t,\Theta(t,x))\nabla_x \Theta(t,x)=I.
$$
Then it follows from \eqref{proposition_3_1_proof_5} that $|(\nabla_\xi^2S)(t,\Theta(t,x))|\gtrsim t$ and thus
\begin{align}
\nonumber
\left|\nabla_x \Theta(t,x)-t^{-1}I\right|
&\le \left|[(\nabla_\xi^2S)(t,\Theta(t,x))]^{-1}\right|\left|I-t^{-1}(\nabla_\xi^2S)(t,\Theta(t,x))\right|\\
\label{proposition_3_1_proof_8}
&\lesssim \begin{cases}
t^{-1}(t+T_1)^{-\rho_{\mathrm L}}
&\text{if}\ \rho_{\mathrm L}<1,\\ t^{-2}\log(tT_1^{-1}+1)&\text{if}\ \rho_{\mathrm L}=1,
\end{cases}
\end{align}
uniformly in $t\ge1$ and $x\in \R^n$ provided $T_1$ is large enough. 
Now we define
$$
\Psi(t,x):=x\cdot \Theta(t,x)-S(t,\Theta(t,x)).
$$
Then \eqref{proposition_3_1_proof_6} yields
\begin{align}
\label{proposition_3_1_proof_9}
\nabla_x\Psi=\Theta+x\cdot \nabla_x\Theta-(\nabla_\xi S)(t,\Theta)\cdot\nabla_x\Theta=\Theta,
\end{align}
which, together with \eqref{proposition_3_1_proof_3} and \eqref{proposition_3_1_proof_6}, shows
\begin{align*}
\partial _t \Psi &= x \cdot \partial_t\Theta   - (\partial _t S)(t,\Theta) - (\nabla _{\xi} S)(t,\Theta)\cdot \partial_t\Theta\\
&= - (\partial _t S)(t,\Theta)
\\ &= - \frac12 | \Theta|^2 - V_{T_1} (t, (\nabla _{\xi} S) (t, \Theta)) \\
&= - \frac{1}{2} | \nabla _x \Psi |^2 - V_{T_1} (t, x).
\end{align*}
Moreover, \eqref{proposition_3_1_proof_7}--\eqref{proposition_3_1_proof_9} implies \eqref{proposition_3_1_1} and \eqref{proposition_3_1_2}. 
\end{proof}


\section{Energy norm of $\CAL{E}(t)$.}
In this section, we estimate the energy norm of $\CAL{E}(t)$ defined in Subsection \ref{subsection_integral_equation}. Recall that the norm $\|\cdot\|_{X_{T,q,r}}$ was defined in Subsection \ref{subsection_integral_equation}. In this section, we omit the subscript $q,r$ and write simply $\|\cdot\|_{X_T}=\|\cdot\|_{X_{T,q,r}}$ for short. Throughout this section, we assume that $\gamma$ and $b$ satisfy the condition in Theorem \ref{theorem_1}. The main result in this section is as follows: 

\begin{proposition}
\label{proposition_4_1}
For sufficiently large $T\ge 1$, any admissible pair $(q,r)$ and any $\ep>0$, 
\begin{align}
\| \CAL{E}\|_{X_T} 
\label{1}
\lesssim T^{b-\min\{\gamma/2,\rho_{\mathrm{L}},\rho_{\mathrm{S}}-1,1\}+\ep} \left(1+\|\widehat{u_+}\|_{H^{\gamma}}^{\frac{2(\lceil \gamma \rceil +1)}{n}}\right)\|\widehat{u_+} \|_{H^{\gamma}}. 
\end{align}
\end{proposition}

The proof  is decomposed into several parts. We begin with the estimates for $\CAL{U}_1(t)$ and $\CAL{U_2}(t)$.

\begin{lemma}
\label{lemma_4_2}
Let $2b<\delta\le 2$, $T>2T_1$, $\ep>0$ and  $(q,r)$ be admissible. Then
\begin{align}
\label{lemma_4_2_1}
\|\CAL{U}_1(t)f\|_{X_T}\lesssim T^{b-\delta/2+\ep}\|f\|_{H^\delta}.
\end{align}
for all $f\in H^\delta$. Moreover, if in addition $\supp f\subset \{x\in \R^n\ |\ |x|\ge c_0\}$, then
\begin{align}
\label{lemma_4_2_2}
\|\CAL{U}_2(t)f\|_{X_T}\lesssim T^{b-\delta/2+\ep}\|f\|_{H^{\delta}}.
\end{align}
\end{lemma}

\begin{proof}Since $\frac 2q=n(\frac12-\frac1r)$, we obtain by \eqref{Dollard_1} and the Hausdorff-Young inequality that
\begin{align}
\nonumber
\|\CAL{U}_1(t)f\|_{L^r}
&\le \|\CAL{D}(t)\F\{1-\CAL{M}(t)\}\F^{-1}f\|_{L^r}\\
\nonumber
&\le |t|^{-n(1/2-1/r)}\|\F\{1-\CAL{M}(t)\}\F^{-1}f\|_{L^r}\\
\nonumber
&\lesssim |t|^{-2/q-\delta/2+1/q+\ep/2} \|\<x\>^{2/q+\ep}|x|^{\delta-2/q-\ep} \F^{-1}f\|_{L^2}\\
\label{lemma_4_2_proof_1}
&\lesssim |t|^{-1/q-\delta/2+\ep/2}\|f\|_{H^{\delta}}
\end{align}
and \eqref{lemma_4_2_1} follows by taking the $L^q([t,\infty))$-norm, multiplying $|t|^b$ and taking the supremum over $t\ge T$ of the both sides of \eqref{lemma_4_2_proof_1}. To show \eqref{lemma_4_2_2}, we write
$$
\CAL{U_2}(t)f
=\chi(x/t)\CAL{M}_{\Psi}(t)\CAL{D}(t) f+\chi(x/t)\CAL{M}_{\Psi}(t)\CAL{D}(t)\F\{\CAL{M}(t)-1\}\F^{-1}f,
$$
where the second term can be dealt by the same argument as for $\CAL{U}_1(t)$ since $|\chi(x/t)\CAL{M}_{\Psi}(t)|\le 1$. Moreover, the first term vanishes identically 
since 
$
\chi(x/t)\CAL{D}(t) f=\CAL{D}(t)\chi f
$ and $\chi f\equiv0$ by the support properties of $\chi$ and $f$. Thus \eqref{lemma_4_2_2} follows. 
\end{proof}


Recall that $\CAL{E}=\CAL{E}_1+\CAL{E}_2$. Proposition \ref{proposition_4_1} then is an immediate consequence of the following two lemmas. In what follows we set for short 
\begin{align}
\label{Gamma}
\Gamma_a(u_+)=\left(1+\|\widehat{u_+}\|_{H^{\gamma}}^{\frac{2a}{n}}\right)\|\widehat{u_+} \|_{H^{\gamma}}.
\end{align}

\begin{lemma}\label{lemma_4_3}
Let $T>2T_1$, $\ep>0$ and $(q,r)$ be admissible. Then 
\begin{align*} 
\|\CAL{E}_1(t)\|_{X_T}\lesssim T^{b-\gamma/2+\ep}\Gamma_{\lceil \gamma \rceil+1}(u_+).
\end{align*}
\end{lemma}

\begin{proof}
By Lemma \ref{lemma_nonlinear}, we have 
\begin{align*}
\|W(t)\|_{H^{\gamma}}
&\lesssim  \left( 1 + (\log t)^{\lceil \gamma \rceil} 
\left\| \widehat{u_+} \right\|_{H^{\gamma}}^{\frac{2\lceil \gamma \rceil}{n}} \right) \left\| \widehat{u_+} \right\|_{H^{\gamma}}
\lesssim (\log t)^{\lr{\gamma}}\Gamma_{\lceil \gamma \rceil}(u_+),\\
\|F(W(t))\|_{H^{\gamma}}&\lesssim (\log t)^{\lr{\gamma}}\Gamma_{\lceil \gamma \rceil+1}(u_+). 
\end{align*}
These two estimates and Lemma \ref{lemma_4_2} with $\delta=\gamma$ show, for $j=1,2$ and $\ep'<\ep$
$$
\|\CAL{U}_j(t)W(t)\|_{X_T}\lesssim T^{b-\gamma/2+\ep'}(\log T)^{\lr{\gamma}}\Gamma_{\lceil \gamma \rceil+1}(u_+)
\lesssim T^{b-\gamma/2+\ep}\Gamma_{\lceil \gamma \rceil+1}(u_+).
$$
Similarly, we know by \eqref{strichartz_3} and \eqref{lemma_4_2_proof_1} with $(q,r)=(\infty,2)$ that
\begin{align*}
\left\| 
\int_t^{\infty} e^{-i (t-s) H} \CAL{U}_j(s) \frac{F(W(s))}{s} ds\right\|_{X_T} 
& \lesssim \sup_{t\ge T}t^b\int_{t}^{\infty} s^{-1}\| \CAL{U}_j(s)F(W(s))\|_{L^2} ds\\
& \lesssim \sup_{t\ge T}t^b\int_{t}^{\infty} s^{-1-\gamma/2+\ep/2}\|F(W(s))\|_{H^{\gamma}} ds\\
&\lesssim T^{b- \gamma/2+\ep} \Gamma_{\lceil \gamma \rceil+1}(u_+)
\end{align*}
for $j=1,2$ and the lemma thus follows. 
\end{proof}

\begin{lemma}\label{lemma_4_4}
For sufficiently large $T$ and any $\ep>0$, $e^{-itH}[i\partial_t,e^{itH}\CAL{U}_3(t)]$ defined on $C_0^\infty((T,\infty)\times \R^{n})$ extends to a bounded operator 
$
\CAL{C}\in \mathbb B(L^\infty([T,\infty);H^{1}),L^1([T,\infty);L^2))
$
which satisfies
\begin{align}
\label{lemma_4_4_2}
\|\CAL{C}(t)f\|_{L^1([T,\infty);L^2)}\lesssim T^{-\min\{\rho_{\mathrm{L}},\rho_{\mathrm{S}}-1,1\}+\ep}\|f\|_{L^\infty([T,\infty);H^{1})}. 
\end{align}
Moreover, for any admissible pair $(q,r)$,
\begin{align} 
\label{K20} 
\|\CAL{E}_2\|_{X_T}  
\lesssim  T^{b-\min\{\rho_{\mathrm{L}},\rho_{\mathrm{S}}-1,1\}+\ep}\Gamma_{1}(u_+).
\end{align}
\end{lemma}

\begin{proof}
We first calculate $e^{-itH}[i\partial_t,e^{itH}\CAL{U}_3(t)]f$ for $f\in C_0^\infty((T,\infty)\times \R^{n})$ and $t\ge T$. Recall that
$$
\CAL{U}_3(t)=(1-\chi_t)U_{\Psi}(t)\F^{-1}=(1-\chi_t)M_\Psi(t)\CAL{M}(-t)e^{-itH_0}\F^{-1},
$$
where $\chi_t(x):=\chi(x/t)$. Since $V^{\mathrm{L}}(x)=V_{T_1}(t,x)$ on $\supp(1-\chi_t)$ for $t\ge 2T_1$, we obtain 
\begin{align}
e^{-itH}i\partial_t e^{itH}(1-\chi_t)U_{\Psi}
\label{lemma_4_4_proof_1}
=(1-\chi_t)(i\partial_t-H_0-V_{T_1})U_{\Psi}-[i\partial_t-H_0,\chi_t] U_{\Psi}-(1-\chi_t)V^{\mathrm{S}}U_{\Psi},
\end{align}
where we have also used the assumption that $\supp V^{\mathrm{sing}}$ is compact to obtain $(1- \chi_t)V^{\mathrm{sing}}\equiv 0$ for sufficiently large $T$. Since
\begin{align*}
(i\partial_t-H_0)\CAL{M}_\Psi(t)
&=\CAL{M}_\Psi(t)\left\{-\partial_t\Psi-\frac12|\nabla\Psi|^2+i(\nabla\Psi)\cdot\nabla +\frac{i}2\Delta\Psi+i\partial_t-H_0\right\},\\
i\partial_t\CAL{M}(-t)e^{-itH_0}&=\CAL{M}(-t)\left\{-\frac{|x|^2}{2t^2}+H_0\right\}e^{-itH_0}+\CAL{M}(-t)e^{-itH_0}i\partial_t\\
&=\left\{-i\frac xt\cdot\nabla-i\frac n{2t}+H_0\right\}\CAL{M}(-t)e^{-itH_0}+\CAL{M}(-t)e^{-itH_0}i\partial_t, 
\end{align*}
we find by using the equation \eqref{K8} that
\begin{align}
\nonumber&(i\partial_t-H_0-V_{T_1})U_{\Psi}\\
\nonumber
&=\CAL{M}_\Psi(t)\left\{-\partial_t\Psi-\frac12|\nabla\Psi|^2-V_{T_1}+i(\nabla\Psi)\cdot\nabla +\frac{i}2\Delta\Psi+i\partial_t-H_0\right\}\CAL{M}(-t)e^{-itH_0}\\
&=i\CAL{M}_\Psi(t)A_{\Psi}(t)\CAL{M}(-t)e^{-itH_0}+U_\Psi i\partial_t
\label{lemma_4_4_proof_2}
=i\CAL{M}_\Psi(t)A_{\Psi}(t)\CAL{D}(t)\F\CAL{M}(t)+U_\Psi i\partial_t,
\end{align}
where 
$$
A_{\Psi}(t)=\left(\nabla \Psi- \frac{x}{t} \right) \cdot \nabla + \frac12 \left( 
\Delta \Psi- \frac{n}{t}\right). 
$$
On $C_0^\infty((T,\infty)\times \R^{n})$, one can open the commutator $[i\partial_t,e^{itH}\CAL{U}_3(t)]$, so \eqref{lemma_4_4_proof_1} and \eqref{lemma_4_4_proof_2} imply
\begin{align}
\nonumber
e^{-itH}[i\partial_t,e^{itH}\CAL{U}_3(t)]f
&=i(1-\chi_t)\CAL{M}_\Psi(t)A_{\Psi}(t)\CAL{D}(t)\F\CAL{M}(t)\F^{-1}f\\
\label{lemma_4_4_proof_3}
&-[i\partial_t-H_0,\chi_t]U_{\Psi}(t)\F^{-1}f-(1-\chi_t)V^{\mathrm{S}}U_{\Psi}(t)\F^{-1}f.
\end{align}

Next, we shall estimate the RHS of \eqref{lemma_4_4_proof_3}. To deal with the first term, we use the formula $$\nabla \CAL{D}(t)\F\CAL{M}(t)\F^{-1}=t^{-1}\CAL{D}(t)\F\CAL{M}(t)\F^{-1}\nabla$$ and Proposition \ref{proposition_3_1} to obtain, for any $\ep>0$, 
\begin{align}
\nonumber
&\|(1-\chi_t)\CAL{M}_\Psi(t)A_{\Psi}(t)\CAL{D}(t)\F\CAL{M}(t)\F^{-1}f(t,\cdot)\|_{L^2}\\
\nonumber
&\lesssim \<t\>^{-\min\{\rho_{\mathrm{L}},1\}+\ep/2}\|\nabla \CAL{D}(t)\F\CAL{M}(t)\F^{-1}f(t,\cdot)\|_{L^2}+\<t\>^{-1-\min\{\rho_{\mathrm{L}},1\}+\ep/2}\|\CAL{D}(t)\F\CAL{M}(t)\F^{-1}f(t,\cdot)\|_{L^2}\\
\label{lemma_4_4_proof_5}
&\lesssim \<t\>^{-1-\min\{\rho_{\mathrm{L}},1\}+\ep/2}\|f(t,\cdot)\|_{H^1}.
\end{align}
To deal with the second term $[i\partial_t-H_0,\chi_t] U_{\Psi}(t)\F^{-1}f$, we compute
\begin{align*}
&[i\partial_t-H_0,\chi_t]\CAL{M}_\Psi(t)\CAL{D}(t)
=\left(-i\frac{x}{t^2}(\nabla \chi)_t+\frac{1}{t}(\nabla \chi)_t\cdot\nabla+\frac{1}{2t^2}(\Delta\chi)_t\right)\CAL{M}_\Psi(t)\CAL{D}(t)\\
&=\left\{\frac{i}{t}\left(\nabla\Psi-\frac xt\right)\cdot (\nabla \chi)_t+\frac{1}{2t^2}(\Delta\chi)_t\right\}\CAL{M}_\Psi(t)\CAL{D}(t)+\frac{1}{t^2}\CAL{M}_\Psi(t)\CAL{D}(t)(\nabla\chi)\cdot \nabla,
\end{align*}
where $(\nabla \chi)_t(x):=\nabla \chi(x/t)$ and $(\Delta\chi)_t(x):=(\Delta\chi)(x/t)$. Hence, we know by Proposition \ref{proposition_3_1} that
\begin{align}
\nonumber
&\|[i\partial_t-H_0,\chi_t]U_{\Psi}(t)\F^{-1}f(t,\cdot)\|_{L^2}\\
\nonumber
&\lesssim (\<t\>^{-1-\min\{\rho_{\mathrm{L}},1\}+\ep/2}+\<t\>^{-2})\|f(t,\cdot)\|_{L^2}+\<t\>^{-2}\|\nabla \F\CAL{M}(t)\F^{-1}f(t,\cdot)\|_{L^2}\\
\label{lemma_4_4_proof_6}
&\lesssim \<t\>^{-1-\min\{\rho_{\mathrm{L}},1\}+\ep/2}\|f(t,\cdot)\|_{H^1}.
\end{align}
For the last term, since $\<x\>^{-1}\lesssim \<t\>^{-1}$ on $\supp (1-\chi_t)$, \eqref{assumption_A_1} implies
\begin{align}
\label{lemma_4_4_proof_7}
\|(1-\chi_t)V^{\mathrm{S}}U_{\Psi}(t)\F^{-1}f(t,\cdot)\|_{L^2}\lesssim \<s\>^{-\rho_{\mathrm{S}}}\|f(t,\cdot)\|_{L^2}. 
\end{align}
It follows from \eqref{lemma_4_4_proof_5}--\eqref{lemma_4_4_proof_7} that
\begin{align}
\label{lemma_4_4_proof_8}
\|e^{-itH}[i\partial_t,e^{itH}\CAL{U}_3(t)]f(t,\cdot)\|_{L^2}\lesssim \<t\>^{-\min\{1+\rho_{\mathrm{L}},\rho_{\mathrm{S}},2\}+\ep/2}\|f(t,\cdot)\|_{H^1}. 
\end{align}
Now we decompose $e^{-itH}[i\partial_t,e^{itH}\CAL{U}_3(t)]=e^{-itH}[i\partial_t,e^{itH}\CAL{U}_3(t)]\<t\>^{\ep/2}\cdot \<t\>^{-\ep/2}$. 
By \eqref{lemma_4_4_proof_8} and H\"older's inequality, we obtain
\begin{align*}
\|e^{-itH}[i\partial_t,e^{itH}\CAL{U}_3(t)]\<t\>^{\ep/2}f\|_{L^1([T,\infty);L^2)}\lesssim T^{-\min\{\rho_{\mathrm{L}},\rho_{\mathrm{S}}-1,1\}+\ep}\|f\|_{L^{3/\ep}([T,\infty);H^1)}.
\end{align*}
A density argument then yields that $e^{-itH}[i\partial_t,e^{itH}\CAL{U}_3(t)]\<t\>^{\ep/2}$ extends uniquely to a bounded operator $\CAL{B}\in \mathbb B(L^{3/\ep}([T,\infty);H^1),L^1([T,\infty);L^2))$ satisfying, for all $f\in L^{3/\ep}([T,\infty);H^{1})$, $$\|\CAL{B}(t)f\|_{L^1([T,\infty);L^2)}\lesssim T^{-\min\{\rho_{\mathrm{L}},\rho_{\mathrm{S}}-1,1\}+\ep}\|f\|_{L^{3/\ep}([T,\infty);H^{1})}.$$Let us define $\CAL{C}(t):=\CAL{B}(t)\<t\>^{-\ep/2}$. Since $\<t\>^{-\ep/2}\in L^{3/\ep}(\R)$, $\CAL{C}(t)$ is  bounded  from $L^\infty([T,\infty);H^{1})$ to $L^1([T,\infty);L^2)$ and satisfies \eqref{lemma_4_4_2}. Finally, using Lemma \ref{lemma_nonlinear}, Strichartz estimates \eqref{strichartz_3} and \eqref{lemma_4_4_2} with $f=W\in L^\infty([T,\infty);H^{1})$ as in the proof of Lemma \ref{lemma_4_3}, we obtain \eqref{K20} for $\CAL{E}_2(t)$.
\end{proof}

\begin{remark}It follows from \eqref{lemma_4_4_proof_8} that, for any $f\in C_0^\infty((T,\infty)\times \R^{n})$, 
$$
\|e^{-itH}[i\partial_t,e^{itH}\CAL{U}_3(t)]f(t,\cdot)\|_{L^1([T,\infty); L^2)}\lesssim T^{-\min\{\rho_{\mathrm{L}},\rho_{\mathrm{S}}-1,1\}+\ep}\|f(t,\cdot)\|_{L^\infty([T,\infty);H^1)}. 
$$
However, this is not enough to conclude that $e^{-itH}[i\partial_t,e^{itH}\CAL{U}_3(t)]$ extends to a bounded operator from $L^\infty([T,\infty);H^1)$ to $L^1([T,\infty); L^2)$ since $C_0^\infty((T,\infty)\times \R^{n})$ is not dense in $L^\infty([T,\infty);H^1)$. Indeed, this is the reason why we introduced an operator $\CAL{B}(t)$ as an intermediate step. Note that $C_0^\infty((T,\infty)\times \R^{n})$ is dense in $L^{3/\ep}([T,\infty);H^1)$. This is verified by approximating $f\in L^{3/\ep}([T,\infty);H^1)$ by a step function $\sum_{j=1}^N a_j\mathds1_{E_j}$ with some $a_j\in H^1$ and bounded intervals $E_j$ and then approximating $a_j$  (resp. $\mathds1_{E_j}$) by $C_0^\infty$-functions  in $H^1$ (resp. $L^{3/\ep}$).  
\end{remark}


\section{Proof of Theorem \ref{theorem_1}}
\label{section_proof_theorem_1}

In this section, we present the proof of Theorem \ref{theorem_1}. To this end, we first prepare a priori estimates of the right hand side of  \eqref{K21} in the energy norm $X_T$. As we have already obtained necessary estimates for $\CAL{E}(t)$ in the previous section, it remains to deal with the term $\CAL{K}[u]$. Let 
$$
(q_n,r_n)
=\begin{cases}(4,\infty)&\text{if $n=1$,}\\(4,4)&\text{if $n=2$,}\\(2,6)&\text{if $n=3$.}\end{cases} 
$$

\begin{lemma}
\label{lemma_5_1}
Suppose $b>\frac n4$ and $u,u_1,u_2\in X(T,R,q_n,r_n)$. Then, for any admissible pair $(q,r)$, 
\begin{align}
\label{lemma_5_1_1}
\|\CAL{K}[u]\|_{X_{T,q,r}}&\lesssim R\left(R^{2/n}T^{-2(b-n/4)/n}+\|\widehat{u_+}\|_{L^\infty}^{2/n}\right),\\
\label{lemma_5_1_2}
\|\CAL{K}[u_1]-\CAL{K}[u_2]\|_{X_{T,q_n,r_n}}&\lesssim \left(R^{2/n}T^{-2(b-n/4)/n}+\|\widehat{u_+}\|_{L^\infty}^{2/n}\right) \left\| u_1-u_2 \right\|_{X_{T,q_n,r_n}}.
\end{align}
\end{lemma}

\begin{proof}
The proof is essentially same as that in \cite{HaWaNa,MMU}. Indeed, if we rewrite $F(u) -F(u_{\mathrm{p}})$ as the sum $F^{(1)}(u)+F^{(2)}(u)$ with $\mathds1_A(x)$ being the characteristic function of $A$ and
\begin{align*}
F^{(1)}(u) &= \mathds1_{\{|u_{\mathrm{p}} | > | u-u_{\mathrm{p}} |\}} \left( F(u) - F(u_{\mathrm{p}}) \right),\\
F^{(2)}(u) &= \mathds1_{\{|u_{\mathrm{p}} | \leq | u-u_{\mathrm{p}} |\}} \left( F(u) - F(u_{\mathrm{p}})\right),
\end{align*}
then $F^{(1)}$ and $F^{(2)}$ satisfy
$$
|F^{(1)}(u)|\lesssim |u_{\mathrm{p}}|^{\frac 2n}|u-u_{\mathrm{p}}|,\quad
|F^{(2)}(u)|\lesssim |u-u_{\mathrm{p}}|^{\frac 2n+1}.
$$
Hence the Strichartz estimate \eqref{strichartz_3} shows that 
$$
\left\|\CAL{K}[u]\right\|_{L^q([t,\infty);L^r)}\lesssim \left\||u_{\mathrm{p}}|^{\frac 2n}|u-u_{\mathrm{p}}|\right\|_{L^1([t,\infty);L^2)}+\left\||u-u_{\mathrm{p}}|^{\frac 2n+1}\right\|_{L^{q_n'}([t,\infty);L^{r_n'})}.
$$
Now one can follow completely the same argument as that in \cite[Section 3]{HaWaNa} for $n=1,2$ and of \cite[Lemma 3.2]{MMU} for $n=3$ to obtain \eqref{lemma_5_1_1}. Since
$$
F(u_1) -F(u_2)=F^{(1)}(u_1)-F^{(1)}(u_2)+F^{(2)}(u_1)-F^{(2)}(u_2),
$$
\eqref{lemma_5_1_2} can be also obtained by the same argument. 
\end{proof}

We are now ready to complete the proof of Theorem \ref{theorem_1}. 

\begin{proof}[Proof of Theorem \ref{theorem_1}]
Since $\frac n4<b<\min\{\frac \gamma2,\rho_{\mathrm{L}},\rho_{\mathrm{S}}-1,1\}$, Proposition \ref{proposition_4_1} and Lemma \ref{lemma_5_1} show that, for any $R>0$ there exists $T,\sigma>0$  such that  \eqref{K21} has a unique solution $u\in X(T,R,q_n,r_n)$ if $\|\widehat{u_+}\|_{L^\infty}\le \sigma$ by the contraction mapping theorem. Proposition \ref{proposition_4_1} and Lemma \ref{lemma_5_1} also show $u \in L^q([T,\infty);L^r)$ for any admissible pair $(q,r)$. Then it follows from Lemma \ref{lemma_integral_1} that $u$ also solves  \eqref{lemma_integral_1_1} with $t_0=T$. Thus, we can apply Proposition \ref{proposition_GWP} below to conclude that $u$ can be extended backward in time uniquely, so $u\in C(\R;L^2)$ is a unique global solution to \eqref{NLS1} satisfying \eqref{theorem_1_1}. 
\end{proof}

\section{Proof of Theorem \ref{theorem_3}}
\label{section_proof_theorem_3}
Here we prove Theorem \ref{theorem_3}. The proof follows basically the same line as that of Theorem \ref{theorem_1}. The only difference is that we use the following proposition instead of Proposition \ref{proposition_3_1}. 

\begin{proposition}
\label{proposition_Dollard_1}
If $t \geq 2$ and $|x|\ge c_0t$ with some $c_0>0$ independent of $t$, then
\begin{align}
\label{proposition_Dollard_1_1}
\left|\nabla \Psi_{\mathrm{D}}(t,x)-\frac xt\right|\lesssim \max\{t^{-\rho_{\mathrm L}},t^{-1} ( \log t )^2 \},\quad 
\left|\Delta \Psi_{\mathrm{D}}(t,x)-\frac nt \right|\lesssim \max\{t^{-1-\rho_{\mathrm L}},t^{-2} ( \log t ) ^2\}.
\end{align}
and
\begin{align}
\label{proposition_Dollard_1_2}
\left|\partial_t \Psi_{\mathrm{D}}(t,x)+\frac12|\nabla\Psi_{\mathrm{D}}(t,x)|^2+V^{\mathrm{L}}(x)\right|\lesssim \max\{t^{-2\rho_{\mathrm L}},t^{-2}(\log t)^4\}.
\end{align}
\end{proposition}

\begin{proof}
As in Section \ref{section_phase}, we may assume $\rho_{\mathrm L}\le1$ without loss of generality. We let $y=x/t$ for short and suppose $|y|\ge c_0$. It follows from \eqref{assumption_A_3} that
\begin{align*}
|(\partial_x^\alpha Q)(t,y)|\lesssim \int_0^t\tau^{|\alpha|}\<\tau y\>^{-\rho_{\mathrm{L}}-|\alpha|}d\tau \le C_\alpha \max\{t^{1-\rho_{\mathrm L}},\log t\}. 
\end{align*}
Differentiating $\widetilde V(t,x)$ in $x$ and using this bound, we also have
\begin{align}
\nonumber
|(\partial_x^\alpha \widetilde V)(t,y)|
&\le C_\alpha\int_0^t \left(\<\tau\>^{-\rho_{\mathrm{L}}}+\frac{|y|}{\tau}\int_0^\tau s^{1+|\alpha|}\<s y\>^{-1-\rho_{\mathrm{L}}-|\alpha|}ds\right)d\tau\\
\label{proposition_Dollard_1_proof_1}
&\le C_\alpha \max\{t^{1-\rho_{\mathrm L}}, ( \log  t ) ^2\}.
\end{align}
Moreover, a direct computation yields
\begin{align}
\nonumber
V^{\mathrm{L}}(x)
&=\frac1t \int_0^t\left\{V^{\mathrm{L}}(\tau y)+\tau y\cdot (\nabla V^{\mathrm{L}})(\tau y)\right\}d\tau\
=\frac1t \left\{Q(t,y)+y\cdot (\nabla Q)(t,y)\right\},\\
\nonumber
\nabla\Psi_{\mathrm{D}}(t,x)&=\frac{x}{t}-\frac1t(\nabla \widetilde{V})(t,y),\\
\label{proposition_Dollard_1_proof_2}
\Delta\Psi_{\mathrm{D}}(t,x)&=\frac{n}{t}-\frac{1}{t^2}(\Delta \widetilde{V})(t,y).
\end{align}
In particular, $\Psi_{\mathrm{D}}$ satisfies
\begin{align}
\nonumber
-\partial_t \Psi_{\mathrm{D}}(t,x)
&=\frac{|x|^2}{2t^2}+\frac1t \left\{Q(t,y)+y\cdot (\nabla Q)(t,y)\right\}-\frac yt \cdot (\nabla \widetilde{V})(t,y)\\
\nonumber
&=\frac12|\nabla \Psi_{\mathrm{D}}(t,x)|^2+\frac yt \cdot (\nabla \widetilde{V})(t,y)-\frac{1}{2t^2}| (\nabla \widetilde{V})(t,y)|^2+V^{\mathrm{L}}(x)-\frac yt \cdot (\nabla \widetilde{V})(t,y)\\
\label{proposition_Dollard_1_proof_3}
&=\frac12|\nabla \Psi_{\mathrm{D}}(t,x)|^2+V^{\mathrm{L}}(x)-\frac{1}{2t^2}|( \nabla  \widetilde{V})(t,y)|^2.
\end{align}
\eqref{proposition_Dollard_1_1} and \eqref{proposition_Dollard_1_2} follows by applying   \eqref{proposition_Dollard_1_proof_1} to the remainder terms of \eqref{proposition_Dollard_1_proof_2} and \eqref{proposition_Dollard_1_proof_3}. 
\end{proof}

\begin{proof}[Proof of Theorem \ref{theorem_3}]
We define the Dollard type modified free evolution $U_{\mathrm{D}}(t)$ by
\begin{align*}
U_{\mathrm{D}}(t) := \CAL{M}_D(t)\CAL{M}(-t)e^{-itH_0},\quad \CAL{M}_D(t)f(x):=e^{i\Psi_{\mathrm{D}}(t,x)}f(x).
\end{align*}
We also define $\CAL{U}_{\mathrm{D},j}(t)$ for $j=1,2,3$ and $\CAL{E}_{\mathrm{D}}(t)=\CAL{E}_{\mathrm{D},1}(t)+\CAL{E}_{\mathrm{D},2}(t)$ by the same way as that for $\CAL{U}_{j}(t)$ and $\CAL{E}(t)=\CAL{E}_{1}(t)+\CAL{E}_{2}(t)$ in Subsection \ref{subsection_integral_equation} with $U_{\Psi}$ replaced by $U_{\mathrm{D}}$. Then we first have
\begin{align*}
\|\CAL{E}_{\mathrm{D},1}(t)\|_{X_T}\lesssim T^{b-\gamma/2+\ep}\Gamma_{\lceil \gamma \rceil+1}(u_+)
\end{align*}
by the completely same proof as that of Lemma \ref{lemma_4_3}. Indeed, we did not use any properties of $\Psi(t,x)$ in the proof of Lemma \ref{lemma_4_3} except for the uniform bound $|e^{i\Psi(t,x)}|=1$. To deal with $\CAL{E}_{\mathrm{D},2}(t)$, we observe by the same calculation as in the proof of Lemma \ref{lemma_4_4} that
\begin{align}
\nonumber
e^{-itH}[i\partial_t,e^{itH}\CAL{U}_{\mathrm{D},3}(t)]
&=i(1-\chi_t)\CAL{M}_{\mathrm{D}}(t)A_{\mathrm{D}}(t)\CAL{D}(t)\F\CAL{M}(t)\F^{-1}\\
\label{theorem_3_proof_1}
&-[i\partial_t-H_0,\chi_t]U_{\mathrm{D}}(t)\F^{-1}-(1-\chi_t)V^{\mathrm{S}}U_{\mathrm{D}}(t)\F^{-1},
\end{align}
where 
$$
A_{\mathrm{D}}(t)=i \left( \partial_t \Psi_{\mathrm{D}}+\frac12|\nabla\Psi_{\mathrm{D}}|^2+V^{\mathrm{L}} \right) +\left(\nabla \Psi_{\mathrm{D}}- \frac{x}{t} \right) \cdot \nabla + \frac12 \left( 
\Delta \Psi_{\mathrm{D}}- \frac{n}{t}\right). 
$$
Since the last two terms of the RHS of \eqref{theorem_3_proof_1} can be dealt with exactly the same way as in Lemma \ref{lemma_4_4}, Proposition \ref{proposition_Dollard_1} shows
$$
\|e^{-itH}[i\partial_t,e^{itH}\CAL{U}_{\mathrm{D},3}(t)]f(t,\cdot)\|_{L^2}\lesssim \<t\>^{-\min\{2\rho_{\mathrm{L}},1+\rho_{\mathrm{L}},\rho_{\mathrm{S}},2\}+\ep}\|f(t, \cdot )\|_{H^1}
$$
for any $t\ge T$ and sufficiently large $T$, and hence
$$
\|\CAL{E}_{\mathrm{D},2}(t)\|_{X_T}\lesssim T^{b-\min\{2\rho_{\mathrm{L}}-1,\rho_{\mathrm{L}},\rho_{\mathrm{S}}-1,1\}+\ep}\Gamma_{1}(u_+).
$$
With these estimates for $\CAL{E}_{\mathrm{D},1}(t)$ and $\CAL{E}_{\mathrm{D},2}(t)$ and Lemma \ref{lemma_5_1} at hand, we can follow the same argument as that in Section \ref{section_proof_theorem_1} in the remaining part of the proof of Theorem \ref{theorem_3}, so we omit it. 
\end{proof}


\appendix

\section{Global existence for the Cauchy problem}
\label{appendix_proposition_GWP}

Here we provide the global existence of the $L^2$-solutions to the Cauchy problem for \eqref{NLS1}. 
\begin{proposition}
\label{proposition_GWP}
Let $u_0\in L^2(\R^n)$ and $t_0\in \R$. Then \eqref{NLS1} with the initial condition $u(t_0)=u_0$ admits a unique global solution $u\in C(\R;L^2(\R^n))\cap L^{4+\frac4n}_{\loc}(\R;L^{2+\frac2n}(\R^n))$ satisfying
\begin{align}
\label{proposition_GWP_1}
u(t)=e^{-i(t-t_0)H}u_0-i\int_{t_0}^te^{-i(t-s)H}F(u(s))ds
\end{align}
and $\norm{u(t)}_{L^2}=\norm{u_0}$ for all $t\in \R$. Moreover, $u\in L^q_{\loc}(\R;L^r(\R^n))$ for any admissible pair $(q,r)$. 
\end{proposition}

\begin{proof}
Suppose $t_0=0$ without loss of generality. The proof relies on the standard argument by Tsutsumi  \cite{Tsutsumi_FE}. Let us set $\norm{v}_{\mathscr Z_T}=\norm{v}_{L^\infty_TL^2_x}+\norm{v}_{L^{2r_0}_TL^{r_0}_x}$ and consider the map
$$
\Phi(v)=e^{-itH}u_0-i\int_0^t e^{-i(t-s)H}F(v(s))ds,
$$
where $L^q_TL^r_x:=L^q([-T,T];L^r(\R^d))$ and $r_0=2+\frac2n$. Strichartz estimates \eqref{strichartz_1} and \eqref{strichartz_2} then imply
\begin{align*}
\norm{\Phi(v)}_{\mathscr Z_T}+\norm{\Phi(v)}_{L^q_TL^{r}_x}\lesssim \norm{u_0}_{L^2}+|\nu|\norm{|v|^{1+2/n}}_{L^{(2r_0)'}_TL^{r_0'}_x}
\lesssim \norm{u_0}_{L^2}+ T^{\delta}\norm{v}_{L^{2r_0}_TL^{r_0}_x}^{1+2/n}.
\end{align*}
for any admissible pair $(q,r)$. It follows from a similar calculation that
$$
\norm{\Phi(v_1)-\Phi(v_2)}_{\mathscr Z_T}\lesssim T^\delta (\norm{v_1}_{\mathscr Z_T}^{2/n}+\norm{v_2}_{\mathscr Z_T}^{2/n})\norm{v_1-v_2}_{\mathscr Z_T}.
$$
Therefore, there exists a unique local solution $u\in C_TL^2_x\cap L^{2r_0}_TL^{r_0}_x$ to \eqref{proposition_GWP_1} with some $T=T(\norm{u_0}_{L^2})$ by the contraction mapping theorem, which belongs to $L^{q}_TL^{r}_x$ for any admissible pair $(q,r)$.

To extend $u$ globally in time, it is enough to observe that, for all $|t|\le T$, 
\begin{align}
\label{appendix_proposition_GWP_1}
\norm{u(t)}_{L^2}=\norm{u_0}_{L^2}.
\end{align}
To this end, we employ the idea by Ozawa  \cite{Ozawa_CVPDE}. Since $e^{-itH}u_0\in L^\infty_tL^2_x\cap L^{2r_0}_TL^{r_0}_x$ and $F(u)\in L^1_TL^2_x\cap L^{(2r_0)'}_TL^{r_0'}_x$ by the above argument, the quantity $\int_0^t \<e^{-isH}u_0,F(u(s))\>ds$ makes sense as the duality coupling on $(L^\infty_tL^2_x\cap L^{2r_0}_TL^{r_0}_x)\times (L^1_tL^2_x+L^{(2r_0)'}_TL^{r_0'}_x)$. Thus, \begin{align}\norm{u(t)}_{L^2}^2
\label{appendix_proposition_GWP_2}=\norm{u_0}_{L^2}^2-2\Im \int_0^t\<e^{-isH}u_0,F(u(s))\>ds+\bignorm{\int_0^t e^{isH}F(u(s))ds}_{L^2}^2,
\end{align}
the second and the last terms  in the RHS of \eqref{appendix_proposition_GWP_2} cancel each other out as follows:
\begin{align*}
\bignorm{\int_0^t e^{isH}F(u(s))ds}_{L^2}^2
&=\Re\int_0^t\int_0^t\<F(u(s)),e^{-i(s-s')H}F(u(s'))\>ds'ds\\
&=2\Re\int_0^t\int_0^s\<F(u(s)),e^{-i(s-s')H}F(u(s'))\>ds'ds\\
&=-2\Im \int_0^t\left\langle F(u(s)),u(s)+i\int_0^se^{-i(s-s')H}F(u(s'))ds'\right\rangle ds\\
&=2\Im \int_0^t\left\langle e^{-isH}u_0,F(u(s))\right\rangle ds,
\end{align*}
where we use the fact $\Im \<F(u),u\>=0$ and \eqref{appendix_proposition_GWP_2}. Hence, \eqref{appendix_proposition_GWP_1} follows. 
\end{proof}

\section{Proof of Lemma \ref{lemma_nonlinear}}
\label{appendix_lemma_nonlinear}
This appendix is devoted to the proof of Lemma \ref{lemma_nonlinear}. We begin with the following lemma. 
\begin{lemma}
\label{lemma_lemma_nonlinear}
Let $f(u)$  be one of $|u|^{\frac 2n}$, $|u|^{\frac 2n-2}u^2$ and $e^{i|u|^{\frac2n}}$. 
\begin{itemize}
\item If $n=1$ and $1/2<s<1$ then for $u,v\in H^s$
\begin{align}
\label{lemma_lemma_nonlinear_1}
\|f(u)v\|_{H^s}\lesssim (1+\|u\|_{H^s}^2)\|v\|_{H^s}.
\end{align}
\item If $n \ge1$, $0<s<\max(\frac n2,1)$ and $\ep>0$, then for $u\in H^{\frac n2+\ep}$ and $v\in H^s$
\begin{align}
\label{lemma_lemma_nonlinear_2}
\|f(u)v\|_{H^s}\lesssim (1+\|u\|_{H^{\frac n2+\ep}}^{\frac 2n})\|v\|_{H^s}.
\end{align}
\end{itemize}
\end{lemma}

\begin{proof}
Let $\dot B_{p,q}^s=\dot B_{p,q}^s(\R^n)$ be the homogeneous Besov space. It is known that
\begin{align}
\label{Besov}
\|u\|_{\dot B_{p,q}^s}\sim \left(\int_0^\infty t^{-1-qs}\sup_{|y|\le t}\|\tau_yu-u\|_{L^p}^qdt\right)^{1/q}
\end{align}
as long as $0<s<1$, where $\tau_y u(\cdot)=u(\cdot-y)$ (see {\it e.g.} \cite[6.3.1 Theorem]{BeLo}). We also have for $u\in H^s$, 
\begin{align}
\label{Besov_2}
\|u\|_{\dot B^s_{2,2}}\sim \|u\|_{\dot H^s}.
\end{align}
At first, we observe that Sobolev embedding implies, for any $\ep>0$, 
\begin{align}
\label{proof_1}
\|f(u)\|_{L^\infty}\le 1+\|u\|_{L^\infty}^{\frac 2n}\lesssim 1+\|u\|_{H^{\frac n2+\ep}}^{\frac 2n},
\end{align}
and that the following well-known inequalities holds for $z_1,z_2\in \C$  (see e.g. Cazenave \cite{Cazenave}): 
\begin{align*}
|f(z_1)-f(z_2)|\lesssim \begin{cases}
\max(|z_1|,|z_2|)^{\frac 2n-1}|z_1-z_2|&\text{if $n=1$,}\\
|z_1-z_2|^{\frac2n}&\text{if $n\ge2$}.
\end{cases}
\end{align*}
Since $\tau_y[f(u)v]-f(u)v=\{f(\tau_yu)-f(u)\}v+f(\tau_yu)(\tau_yv-v)$, this inequality implies
\begin{align}
\label{proof_2}
|\tau_y[f(u)v]-f(u)v|\lesssim \begin{cases}\|u\|_{L^\infty}|\tau_yu-u||v|+\|f(u)\|_{L^\infty}|\tau_y v-v|&\text{if}\ n=1,\\
|\tau_yu-u|^{\frac 2n}|v|+\|f(u)\|_{L^\infty}|\tau_y v-v|&\text{if}\ n\ge2.\end{cases}
\end{align}

Now we first suppose $n=1$ and $1/2<s<1$. Substituting \eqref{proof_2} with $n=1$ into \eqref{Besov} and using H\"older's inequality, \eqref{Besov_2}, \eqref{proof_1}  and the embedding $H^s\hookrightarrow L^\infty$ show \eqref{lemma_lemma_nonlinear_1} as follows: 
\begin{align*}
\|f(u)v\|_{H^s}
&\lesssim \|f(u)\|_{L^\infty}\|v\|_{L^2}+\|u\|_{L^\infty}\|u\|_{\dot B_{2,2}^s}\|v\|_{L^\infty}+\|f(u)\|_{L^\infty}\|v\|_{\dot B_{2,2}^s}\\
&\lesssim (1+\|u\|_{H^s}^2)\|v\|_{H^s}. 
\end{align*}

Next, if $n=1$ and $0<s<1/2$, then since $1/s>2$ we have the continuous embeddings $H^s\hookrightarrow L^{\frac{2}{1-2s}}$ and $H^{1/2}\hookrightarrow \dot W^{s,1/s}\hookrightarrow \dot B^s_{1/s,2}$. Moreover, H\"older's inequality implies
$$
\|(\tau_yu-u)v\|_{L^2}\le \|\tau_yu-u\|_{L^{\frac 1s}}\|v\|_{L^{\frac{2}{1-2s}}}\lesssim \|\tau_yu-u\|_{L^{\frac 1s}}\|v\|_{H^s}
$$
Plugging \eqref{proof_2} into \eqref{Besov} and using this estimate, we similarly obtain as above that, for any $\ep>0$, 
\begin{align*}
\|f(u)v\|_{H^s}&\lesssim \|f(u)\|_{L^\infty}\|v\|_{L^2}+\|u\|_{L^\infty}\|u\|_{\dot B_{1/s,2}^s}\|v\|_{H^s}+\|f(u)\|_{L^\infty}\|v\|_{\dot H^s}\\&\lesssim (1+\|u\|_{H^{\frac 12+\ep}}^2)\|v\|_{H^s}. 
\end{align*}

Finally, if $n\ge2$ and $0<s<\max(\frac n2,1)$, then we similarly obtain the continuous embeddings $H^s\hookrightarrow L^{\frac{2n}{n-2s}}$ and $H^{n/2}\hookrightarrow \dot W^{s,2/s}\hookrightarrow \dot B^s_{2/s,2}$ and, by H\"older's inequality, 
$$
\||\tau_yu-u|^{\frac 2n}v\|_{L^2}\le \|\tau_yu-u\|_{L^{\frac 2s}}^{\frac 2n}\|v\|_{L^{\frac{2n}{n-2s}}}\lesssim \|\tau_yu-u\|_{L^{\frac 2s}}^{\frac 2n}\|v\|_{H^s}.
$$
Thus, we similarly obtain, for any $\ep>0$, 
\begin{align*}
\|f(u)v\|_{H^s}\lesssim \|f(u)\|_{L^\infty}\|v\|_{L^2}+\|u\|_{\dot B_{2/s,2}^s}^{\frac 2n}\|v\|_{H^s}+\|f(u)\|_{L^\infty}\|v\|_{\dot H^s}\lesssim (1+\|u\|_{H^{\frac n2+\ep}}^{\frac 2n})\|v\|_{H^s}.
\end{align*}
This completes the proof
\end{proof}

\begin{proof}[Proof of Lemma \ref{lemma_nonlinear}]
We first show \eqref{lemma_nonlinear_1}. To this end, we may assume $\lambda=1$ without loss of generality since the case $\lambda=1$ implies, for $\lambda\ge1$, 
\begin{align*}
\|e^{i\lambda|u|^{\frac2n}}u\|_{H^\gamma}
=\lambda^{-\frac2n}\|e^{i|\lambda^{\frac{n}{2}}u|^{\frac2n}}\lambda^{\frac{n}{2}}u\|_{H^\gamma}
\lesssim \big(1+\lambda^{\lceil \gamma \rceil}\|u\|_{H^\gamma}^{\frac{2\lceil \gamma \rceil}{n}}\big)\|u\|_{H^\gamma}.
\end{align*}
The proof is divided into the following four cases: (i) $n=1$ and $\frac12<\gamma<1$; (ii) $n=1$ and $\gamma=1$; (iii) $n=1$ and $\frac32<\gamma<2$; (iv) $n\ge1$ and $\max(1,\frac n2)<\gamma<\min(2,1+\frac n2,1+\frac 2n)$.

(i) If $n=1$ and $\frac12<\gamma<1$, then \eqref{lemma_nonlinear_1} directly follows from\eqref{lemma_lemma_nonlinear_1} with $v=u$. 

(ii) If $n=1$ and $\gamma=1$, then \eqref{lemma_nonlinear_1} is a direct consequence of the formula
\begin{align}
\label{proof_3}
\nabla(e^{i|u|^{\frac2n}}u)
=e^{i|u|^{\frac2n}}\left(\frac in |u|^{\frac2n}\nabla u+\frac in|u|^{\frac2n-2}u^2 \overline{\nabla u}+\nabla u\right).
\end{align}

(iii) Let $n=1$, $\frac32<\gamma<2$. Using \eqref{lemma_lemma_nonlinear_1} twice with $f(u)=e^{i|u|^{2}}$ and $v=|u|^{2}\nabla u$ in the first place and with $f(u)=|u|^{2}$ and $v=\nabla u$ in the second place shows
\begin{align*}
\|e^{i|u|^{2}}|u|^{2}\nabla u\|_{H^{\gamma-1}}
\lesssim (1+\|u\|_{H^{\gamma-1}}^2)\||u|^{2}\nabla u\|_{H^{\gamma-1}}
\lesssim (1+\|u\|_{H^{\gamma }}^4)\|u\|_{H^\gamma}.
\end{align*}
Similarly, \eqref{lemma_lemma_nonlinear_1} also implies
\begin{align*}
\|e^{i|u|^{\frac2n}}|u|^{\frac2n-2}u^2\overline{\nabla u}\|_{H^{\gamma-1}}
&\lesssim (1+\|u\|_{H^{\gamma }}^4)\|u\|_{H^\gamma},\\
\|e^{i|u|^{\frac2n}}\nabla u\|_{H^{\gamma-1}}
&\lesssim (1+\|u\|_{H^{\gamma}}^2)\|u\|_{H^\gamma}.
\end{align*}
Plugging these three estimates into \eqref{proof_3} yields \eqref{lemma_nonlinear_1}. 

(iv) Let $n\ge1$, $\max(1,\frac n2)<\gamma<\min(2,1+\frac 2n,1+\frac n2)$. Then \eqref{lemma_nonlinear_1} can  be verified by the same argument as in (iii) with \eqref{lemma_lemma_nonlinear_2} instead of \eqref{lemma_lemma_nonlinear_1}. Indeed, taking $\ep>0$ such that $\frac n2+\ep<\gamma$, we have
\begin{align*}
\|e^{i|u|^{\frac2n}}|u|^{\frac2n}\nabla u\|_{H^{\gamma-1}}
\lesssim  (1+\|u\|_{H^{\frac n2+\ep}}^{\frac 2n})\||u|^{\frac 2n}\nabla u\|_{H^{\gamma-1}} 
\lesssim (1+\|u\|_{H^{\gamma }}^{\frac 4n})\|u\|_{H^\gamma}.
\end{align*}
The remaining part is analogous and we omit it. This completes the proof of \eqref{lemma_nonlinear_1}. 

Finally, in order to prove \eqref{lemma_nonlinear_2}, we recall the following estimate: 
$$
\norm{|v|^{n/2}v}_{\dot H^\gamma}\lesssim \norm{v}_{L^\infty}^{n/2}\norm{v}_{\dot H^\gamma}
$$
(see, e.g., Ginibre-Ozawa-Velo \cite[Lemma 3.4]{GOV1994}). Let $v=e^{i\lambda |u|^{\frac2n}}u$ so that $e^{i\lambda |u|^{\frac2n}} F(u)=\nu|v|^{2/n}v$. Then this estimate and \eqref{lemma_nonlinear_1} with $\rho=0$ imply
\begin{align*}
\|e^{i\lambda |u|^{\frac2n}} F(u)\|_{H^\gamma}
&\lesssim \norm{|v|^{2/n}v}_{L^2}+\norm{|v|^{2/n}v}_{\dot H^\gamma}\\
&\lesssim \norm{v}_{L^\infty}^{2/n}\norm{v}_{H^\gamma}
\lesssim \big(1+\lambda^{\lceil \gamma \rceil}\norm{u}_{H^\gamma}^{\frac{2\lceil \gamma \rceil}{n}}\big)
\norm{u}_{L^{\infty}}^{\frac{2}{n}} \| u\|_{H^\gamma}.
\end{align*}
This completes the proof of the proposition. 
\end{proof}


\section*{Acknowledgments}
M. K. is partially supported by JSPS KAKENHI Grant Number 20K14328. H. M. is partially supported by JSPS KAKENHI Grant Number JP21K03325. 

\end{document}